\documentclass{article}
\usepackage[utf8]{inputenc}
\usepackage{graphicx}
\usepackage{amsfonts}
\usepackage{amsmath}
\usepackage{amssymb}
\usepackage{fancyhdr}
\usepackage{titlesec}
\usepackage{indentfirst}
\usepackage{booktabs}
\usepackage{verbatim}
\usepackage{color}
\usepackage{amsthm}
\usepackage{subfigure}
\usepackage{mathabx}

\usepackage[page,header]{appendix}
\usepackage{titletoc}

\newcommand{\rd}{\,\mathrm{d}}
\numberwithin{equation}{section}
\newtheorem{theorem}{Theorem}[section]
\newtheorem{lemma}[theorem]{Lemma}

\newtheorem{definition}[theorem]{Definition}
\newtheorem{remark}[theorem]{Remark}

\usepackage{mathtools}

\def\bx{{\bf x}}
\def\by{{\bf y}}
\def\bz{{\bf z}}

\def\cM{\mathcal{M}}

\def\cF{\mathcal{F}}

\def\sgn{\textnormal{sgn}}
\def\supp{\textnormal{supp\,}}
\def\diam{\textnormal{diam\,}}

\def\cpf{c_{\textnormal{PF}}}
\def\tcpf{\tilde{c}_{\textnormal{PF}}}

\begin{document}
	
\title{Extended convexity and uniqueness of minimizers for interaction energies}

\author{ Ruiwen Shu\footnote{Department of Mathematics, University of Georgia, Athens, GA 30602 (ruiwen.shu@uga.edu).}}

\maketitle

\begin{abstract}
	Linear interpolation convexity (LIC) has served as the crucial condition for the uniqueness of interaction energy minimizers. We introduce the concept of the LIC radius which extends the LIC condition. Uniqueness of minimizer up to translation can still be guaranteed if the LIC radius is larger than the possible support size of any minimizer. Using this approach, we obtain uniqueness of minimizer for power-law potentials $W_{a,b}(\bx) = \frac{|\bx|^a}{a} - \frac{|\bx|^b}{b},\,-d<b<2$ with $a$ slightly smaller than 2 or slightly larger than 4. The estimate of LIC radius for $a$ slightly smaller than 2 is done via a Poincar\'{e}-type inequality for signed measures. To handle the case where $a$ slightly larger than 4, we truncate the attractive part of the potential at large radius and prove that the resulting potential has positive Fourier transform. We also propose to study the logarithmic power-law potential $W_{b,\ln}(\bx) = \frac{|\bx|^b}{b}\ln|\bx|$. We prove its LIC property for $b=2$ and give the explicit formula for minimizer. We also prove the uniqueness of minimizer for $b$ slightly less than 2 by estimating its LIC radius.
\end{abstract}


\section{Introduction}\label{sec_intro}

In this paper we study the minimizers of the interaction energy 
\begin{equation}\label{E}
	E[\rho] = \frac{1}{2}\int_{\mathbb{R}^d}\int_{\mathbb{R}^d}W(\bx-\by)\rd{\rho(\by)}\rd{\rho(\bx)}\,,
\end{equation}
where $\rho\in \cM(\mathbb{R}^d)$ is a probability measure. $W:\mathbb{R}^d\rightarrow \mathbb{R}\cup\{\infty\}$ is an interaction potential, which satisfies the basic assumption
\begin{equation}\begin{split}
		\text{{\bf (W0)}: $W$} & \text{ is locally integrable, bounded from below} \\ & \text{and lower-semicontinuous, with $W(\bx)=W(-\bx)$.}
\end{split}\end{equation}
As a consequence, $E[\rho]$ is always well-defined for any $\rho\in \cM(\mathbb{R}^d)$, taking values in $\mathbb{R}\cup \{\infty\}$.

The existence of minimizers of $E$ is well-understood: (see also similar results in \cite{CCP15,CFT15})
\begin{theorem}[\cite{SST15}]\label{thm_exist2}
	Assume {\bf (W0)}. Assume $\lim_{|\bx|\rightarrow\infty}W(\bx)=:W_\infty$ is a real number or $\infty$. If there exists $\rho\in \cM(\mathbb{R}^d)$ such that $E[\rho]< \frac{1}{2}W_\infty$, i.e., 
	\begin{equation}\label{thm_exist2_1}
		\inf_{\rho\in\cM(\mathbb{R}^d)}E[\rho] < \frac{1}{2}W_\infty\,,
	\end{equation}
	then there exists a minimizer of $E$.
\end{theorem}
Under a mild extra assumption on $W$, \cite{CCP15} showed that any minimizer of $E$ is compactly supported.

In this paper we focus on the \emph{uniqueness} of minimizers up to translation. The uniqueness of minimizers has been studied by many recent works \cite{Lop19,CS21,ST,DLM1,DLM2,Fra}. The main tool to study the uniqueness of minimizers is the \emph{linear interpolation convexity (LIC)}. We say $W$ has LIC, if for any distinct compactly-supported $\rho_0,\rho_1\in\cM(\mathbb{R}^d)$ such that $\int_{\mathbb{R}^d}\bx\rho_0(\bx)\rd{\bx}=\int_{\mathbb{R}^d}\bx\rho_1(\bx)\rd{\bx}$, the function $t\mapsto E [(1-t)\rho_0+t\rho_1]$ is strictly convex on $t\in [0,1]$. For LIC potentials, the energy minimizer is necessarily unique (up to translation). The LIC property is closely related to the positivity of $\hat{W}$, the Fourier transform of $W$, due to the formal Fourier representation formula \eqref{EFou} that will be justified for some potentials. 

One important family of attractive-repulsive interaction potentials are the power-law potentials
\begin{equation}\label{Wab}
	W_{a,b}(\bx) = \frac{|\bx|^a}{a} - \frac{|\bx|^b}{b}\,,
\end{equation}
where $-d < b < a$. Here we adopt the tradition that $\frac{|\bx|^0}{0}$ is understood as $\ln|\bx|$. The associated interaction energy given by \eqref{E} will be denoted as $E_{a,b}$. For these potentials, it is known that LIC holds for $-d<b\le 2,\,2\le a \le 4,\,b<a$, see \cite{Lop19,CS21,Fra}, except for the case $(a,b)=(4,2)$ which satisfies a non-strict LIC, see \cite{DLM2}.

Convexity theory is also crucial for obtaining explicit formula for minimizers. In fact, it is known that any minimizer satisfies the Euler-Lagrange condition \cite{BCLR13_1}, see Theorem \ref{thm_EL}, which is a necessary condition. However, if LIC holds, then it is also a sufficient condition \cite{CS21}. Therefore, for LIC potentials, one can prove that some prescribed distribution is the unique minimizer as long as the Euler-Lagrange condition is verified, which only requires explicit calculation. Explicit formulas for the minimizer were obtained in this way for some power-law potentials, which include minimizers as a locally integrable function \cite{CH17,CS21,Fra,Shu_expli}, spherical shells \cite{BCLR13_2,DLM1,FM} and discrete measures \cite{DLM2}. Minimizers for some anisotropic interaction energies in 2D and 3D were also derived from this approach \cite{MRS,CMMRSV1,CMMRSV2,MMRSV1,MMRSV2,CS_2D,CS_3D}.

Despite the success of the LIC property in the study of uniqueness and explicit formula of minimizers, it is not a necessary condition for the uniqueness. To the best of the author's knowledge, the only existing result on the uniqueness of minimizer for non-LIC potentials is \cite{DLM2}. In this paper, it was proved that the sum of equal Dirac masses, located at the vertices of a unit simplex, is the unique minimizer for $W_{a,b}$ if $a\ge \min\{2+d,4\}$, $b\ge 2$, $b<a$. In this case the minimizer is unique up to translation and \emph{rotation}, which has a different nature from the uniqueness derived from the LIC condition.

In this paper we aim to extend the LIC condition so that the uniqueness of minimizer up to translation can be proved for a wider class of potentials. The first two main results of this paper give the uniqueness of minimizers for power-law potentials $W_{a,b}$ with extended ranges of $(a,b)$, stated as follows.
\begin{theorem}\label{thm_main1}
	Consider the interaction potential $W_{a,b}$ in \eqref{Wab} with $-d<b<a$ and the associated energy $E_{a,b}$ given by \eqref{E}. For any $-d<b<2$, there exists $a_-(b)<2$, such that the minimizer of $E_{a,b}$ is unique up to translation if $a_-(b)<a<2$.
\end{theorem}
\begin{theorem}\label{thm_main2}
	Consider the interaction potential $W_{a,b}$ in \eqref{Wab} with $-d<b<a$ and the associated energy $E_{a,b}$ given by \eqref{E}. For any $-d<b<2$, there exists $a_+(b)>4$, such that the minimizer of $E_{a,b}$ is unique up to translation if $4<a<a_+(b)$.
\end{theorem}

Another related potential is the logarithmic power-law potential
\begin{equation}\label{Wbln}
	W_{b,\ln}(\bx) = \frac{|\bx|^b}{b}\ln|\bx|\,,
\end{equation}
with $b>-d,\,b\ne 0$, which can be viewed as the asymptotic limit as $a\rightarrow b_+$ for $W_{a,b}$ under suitable rescaling. The associated interaction energy given by \eqref{E} will be denoted as $E_{b,\ln}$. We first show that the logarithmic power-law potential $W_{2,\ln}(\bx)=\frac{|\bx|^2}{2}\ln|\bx|$ is LIC and derive the explicit formula for its associated energy minimizer.
\begin{theorem}\label{thm_Wbln}
	The logarithmic power-law potential $W_{2,\ln}(\bx)=\frac{|\bx|^2}{2}\ln|\bx|$ is LIC. Its associated energy minimizer is unique (up to translation), given by
	\begin{equation}
		\rho(x) = C(R^2-x^2)_+^{-1/2},\quad R = \exp\Big[-\frac{1}{2} + \frac{\rd}{\rd{a}}\Big|_{a=2} \Big( \frac{\Gamma(\frac{3-a}{2})\sin\big((a-1)\frac{\pi}{2}\big)}{\Gamma(\frac{4-a}{2})(a-1)\sqrt{\pi}}\Big)\Big]\,,
	\end{equation}
	for $d=1$ where $C$ is a normalizing constant, and
	\begin{equation}
		\rho(x) = \delta_{\partial B(0;R)},\quad R = \exp\Big[-\frac{1}{2} + \frac{\rd}{\rd{a}}\Big|_{a=2} \Big(\frac{\Gamma(\frac{d+1}{2})\Gamma(\frac{2d+a-2}{2})}{2^{a-2}\Gamma(\frac{d+a-1}{2})\Gamma(d)}\Big)\Big]\,,
	\end{equation}
	for $d\ge 2$ where $\delta_{\partial B(0;R)}$ denotes the uniform distribution on $\partial B(0;R)$ with total mass 1.
\end{theorem}

Similar result for the uniqueness of minimizer is also obtained for $W_{b,\ln}$ when $b$ is slightly less than 2 (for which LIC does not hold).
\begin{theorem}\label{thm_main3}
	Consider the interaction potential $W_{b,\ln}$ in \eqref{Wbln} with $b>-d,\,b\ne 0$ and the associated energy $E_{b,\ln}$ given by \eqref{E}. There exists $b_*<2$, such that the minimizer of $E_{b,\ln}$ is unique up to translation if $b_*<b\le 2$.
\end{theorem}

In the rest of this section we sketch the proof of our main results.

\subsection{The LIC radius and related concepts}
The main concept involved in the proof is the \emph{LIC radius}, which extends the LIC condition, stated as follows.
\begin{definition}
	Consider an interaction potential $W$ satisfying {\bf (W0)}. We define the \emph{LIC radius} of $W$ as a number in $[0,\infty]$ given by
	\begin{equation}\begin{split}
		R_{\textnormal{LIC}}[W] = & \sup\Big\{R\ge 0: \textnormal{ for any distinct compactly-supported $\rho_0,\rho_1\in\cM(\overline{B(0;R)})$} \\
		& \textnormal{such that $E[\rho_i]<\infty,\,i=0,1$ and $\int_{\mathbb{R}^d}\bx\rd{\rho_0(\bx)}=\int_{\mathbb{R}^d}\bx\rd{\rho_1(\bx)}=0$, the function} \\
		& \textnormal{$t\mapsto E [(1-t)\rho_0+t\rho_1]$ is strictly convex on $t\in [0,1]$}\Big\}\,.
	\end{split}\end{equation}
\end{definition}
It is clear that an LIC potential $W$ has $R_{\textnormal{LIC}}[W]=\infty$. {The condition $R_{\textnormal{LIC}}[W]=0$ can almost imply the concept of infinitesimal concavity \cite[Section 7]{CS21}, which implies that any superlevel set of any minimizer does not have an interior point. To be precise, if $R_{\textnormal{LIC}}[W]=0$, then for any $\epsilon>0$, one can find distinct $\rho_0,\rho_1\in\cM(\overline{B(0;\epsilon)})$, such that $E[\rho_i]<\infty,\,i=0,1$ and $\int_{\mathbb{R}^d}\bx\rd{\rho_0(\bx)}=\int_{\mathbb{R}^d}\bx\rd{\rho_1(\bx)}=0$, but $t\mapsto E [(1-t)\rho_0+t\rho_1]$ is not strictly convex on $t\in [0,1]$. If one also has $E[\rho_0+\rho_1]<\infty$, then by calculating $\frac{\rd^2}{\rd{t}^2}E [(1-t)\rho_0+t\rho_1] = 2E[\rho_1-\rho_0]$, we obtain a signed measure $\mu=\rho_1-\rho_0$ satisfying $\supp\mu\subset \overline{B(0;\epsilon)}$, $\int_{\mathbb{R}^d} \rd{\mu}=\int_{\mathbb{R}^d} \bx\rd{\mu}(\bx)=0$, $E[|\mu|]<\infty$ and $E[\mu] \le 0$. Suppose one slightly strengthens this statement by replacing $E[\mu] \le 0$ with $E[\mu]<0$, then the same is true for a slight mollification of $\mu$ (by Lemma 3.2, under some mild assumptions on $W$), and then one can deduce infinitesimal concavity.}

If a potential $W$ satisfies $0<R_{\textnormal{LIC}}[W]<\infty$, then one can use linear interpolation arguments as if $W$ is LIC, as long as the probability measures involved in the argument are supported inside $B(0;R_{\textnormal{LIC}}[W])$.

It is known from \cite{CCP15} that under mild assumptions on $W$ (see Theorem 1.4 therein), any minimizer is compactly supported, and for a given $W$, one can get an upper bound of the size of any minimizer explicitly. We therefore formulate the following definition.
\begin{definition}
	Consider an interaction potential $W$ satisfying the assumptions of Theorem \ref{thm_exist2}. We say $R_*\in (0,\infty)$ is an \emph{upper bound of minimizer size} for $W$, if the following holds:
	\begin{itemize}
		\item Any minimizer for $W$ is compactly supported.
		\item Any minimizer $\rho$ for $W$ with $\int_{\mathbb{R}^d}\bx\rd{\rho(\bx)}=0$ satisfies $\supp\rho\subset \overline{B(0;R_*)}$.
	\end{itemize}
\end{definition}
Suppose an upper bound of minimizer size for $W$ is known, say $R_*$, then one does not necessarily need the LIC property, i.e., $R_{\textnormal{LIC}}[W]=\infty$, to guarantee the uniqueness of minimizer. One only needs a weaker condition $R_{\textnormal{LIC}}[W]> R_*$, due to the following lemma whose proof is straightforward.
\begin{lemma}\label{lem_LICR}
	Assume $W$ satisfies the assumptions of Theorem \ref{thm_exist2}. If $R_*$ is an upper bound of minimizer size for $W$ and $R_{\textnormal{LIC}}[W]> R_*$, then the energy minimizer for $W$ is unique up to translation.
\end{lemma}

\begin{proof}
	By assumption, we have the existence of minimizers, that any minimizer is compactly supported, and that any minimizer $\rho$ with $\int_{\mathbb{R}^d}\bx\rd{\rho(\bx)}=0$ satisfies $\supp\rho\subset \overline{B(0;R_*)}$. Suppose there are two distinct minimizers $\rho_0,\rho_1$ which are not translations of each other, then we may assume $\int_{\mathbb{R}^d}\bx\rho_0(\bx)\rd{\bx}=\int_{\mathbb{R}^d}\bx\rho_1(\bx)\rd{\bx}=0$ without loss of generality, by translation. It follows that $\supp\rho_i\subset \overline{B(0;R_*)},\,i=0,1$. Then the assumption $R_{\textnormal{LIC}}[W]> R_*$ implies that $t\mapsto E [(1-t)\rho_0+t\rho_1]$ is strictly convex on $t\in [0,1]$. Since $\rho_0$ and $\rho_1$ are minimizer, we have $E[\rho_0]=E[\rho_1]$, and thus the probability measure $\frac{1}{2}\rho_0+\frac{1}{2}\rho_1$ has a smaller energy, contradicting the minimizing property of $\rho_0,\rho_1$.	
\end{proof}

We will  show in Section \ref{sec_LICRER} that the sufficiency of the Euler-Lagrange condition can also be extended if $R_{\textnormal{LIC}}[W]$ is sufficiently large.

\subsection{Estimating the LIC radius}

To obtain uniqueness of minimizer for potentials beyond the LIC class, we seek for those $W$ whose LIC radius is sufficiently large so that one can apply Lemma \ref{lem_LICR}. In particular, we fix $-d<b<2$ and view the power-law potential $W_{a,b}$ in \eqref{Wab} as a one-parameter family in $a$. Since we know that $R_{\textnormal{LIC}}[W_{a,b}] = \infty$ for $2\le a \le 4$, it is reasonable to expect that $R_{\textnormal{LIC}}[W_{a,b}]$ is large if $a$ is slightly smaller than 2 or slightly greater than 4 so that the uniqueness of minimizer can be proved.

The last statement is indeed true but the proof is not as easy as one might expect. In fact, one cannot simply proceed by a standard perturbative argument, because the functional $R_{\textnormal{LIC}}$ may not behave nicely along a one-parameter family of potentials in general (as discussed in the next subsection).

To analyze the LIC radius for $W_{a,b}$ with $-d<b<a<2$, we switch to the Fourier representation \eqref{EFou} of the energy, which is justified in Section \ref{sec_Fou} in a general framework. Then the problem is converted into an estimate between two Sobolev norms of a compactly supported signed measure, see Theorem \ref{thm_powercomp}, which we believe is interesting by itself. This estimate is proved via the aid of smoothing techniques and the Heisenberg uncertainty principle. This leads to the proof of Theorem \ref{thm_main1}, as well as a characterization of the LIC radius of $W_{a,b}$ in terms of the optimal constant in Theorem \ref{thm_powercomp} (see Remark \ref{rmk_sharp}). We also prove Theorem \ref{thm_powercomplog}, a logarithmic analogue of Theorem \ref{thm_powercomp}, which gives the proof of Theorem \ref{thm_main3}.

The LIC radius for $W_{a,b}$ with $-d<b<2$ and $a>4$ is more tricky because the energy from the attractive part $\frac{|\bx|^a}{a}$ does not admit a Fourier representation as in \eqref{EFou}. Therefore we switch to a different strategy. To prove that $R_{\textnormal{LIC}}[W_{a,b}]$ is at least $R$ for some fixed number $R$, one is allowed modifying $W_{a,b}$ arbitrarily at $|\bx|>2R$. When $a$ is sufficiently close to 4, we proved that a smooth truncation of the attractive part of $W_{a,b}$ at a large radius will make its Fourier transform positive, and then Theorem \ref{thm_main2} follows.

\subsection{Some discontinuous behavior of the LIC radius}

In this subsection we give a few examples in which the functional $R_{\textnormal{LIC}}$ does not always behave nicely along a one-parameter family of potentials. These examples show that one has to be extremely careful when using the concept of the LIC radius.

One example is to consider the power-law potentials $W_{a,b}$ in \eqref{Wab} with fixed $2<a<4$ and $b$ being the parameter. When $-d<b\le 2$, one has $R_{\textnormal{LIC}}[W_{a,b}]=\infty$. However, for any $2<b<a$, one has $R_{\textnormal{LIC}}[W_{a,b}]=0$. In $d=1$, this can be seen by taking $\mu = \delta_\epsilon - 2\delta_0 + \delta_{-\epsilon}$ (c.f. {the signed measure $\mu$ in \eqref{thm_RLICab_1} of} Theorem \ref{thm_RLICab}) and calculating
\begin{equation}
	E_{a,b}[\mu] = -\frac{1}{b}\big((2\epsilon)^b-4\epsilon^b\big) + \frac{1}{a}\big((2\epsilon)^a-4\epsilon^a\big) = -\frac{2^b-4}{b}\epsilon^b + \frac{2^a-4}{a}\epsilon^a < 0\,,
\end{equation}
for any sufficiently small $\epsilon>0$. Higher dimensional cases can be proved similarly. This is indeed a reminiscent of \cite{CFP17} which proved that any minimizer for $W_{a,b}$ with $2<b<a$ is a finite linear combination of Dirac masses.

Another example is $W_{a,b}$ with $b=2$ and $a>2$ being the parameter. When $2<a<4$, one has $R_{\textnormal{LIC}}[W_{a,b}]=\infty$. However, for $a\ge 4$, one has $R_{\textnormal{LIC}}[W_{a,b}]=0$. In $d=1$, this can be seen by taking $\mu = \delta_0 - 3\delta_\epsilon + 3\delta_{2\epsilon}-\delta_{3\epsilon}$ and calculating
\begin{equation}
	E_{a,2}[\mu] = \frac{\epsilon^a}{a}(-15 + 6\cdot 2^a - 3^a)\,.
\end{equation}
The last quantity is zero when $a=4$, and it is not hard to show that it is negative for any $a>4$. Higher dimensional cases can be proved similarly.

One last example is the anisotropic interaction potential $W_\alpha = \frac{|\bx|^{-s}}{s}(1+\alpha\Omega(\frac{\bx}{|\bx|})) + \frac{|\bx|^2}{2}$  in dimension $d=2$, where $0\le s < 1$ and $\Omega$ is a smooth angular profile satisfying $\min \Omega = 0$. It is known that there exists a critical value $\alpha_*>0$, such that $R_{\textnormal{LIC}}[W_\alpha]=\infty$ for $0\le \alpha \le \alpha_*$ and { infinitesimally concave for $\alpha > \alpha_*$, see \cite[Proposition 2.5]{CS_2D}. One also has  $R_{\textnormal{LIC}}[W_\alpha]=0$ for $\alpha > \alpha_*$. To see this, one needs to modify the proof of \cite[Proposition 2.5]{CS_2D} in the following way: Notice that the signed measure $\mu_1$ therein has vanishing moments of all order, since $0\notin \supp\hat{\mu}_1$. When we use $\mu_1$ to construct the compactly-supported signed measure $\mu\in L^\infty$ satisfying $\int_{\mathbb{R}^d}\rd{\mu}$ , we subtract one more term to enforce the extra condition $\int_{\mathbb{R}^d} \bx\rd{\mu}(\bx)=0$. Then one can proceed similarly to prove $E[\mu]<0$, using the condition $\int_{\mathbb{R}^d}\bx\rd{\mu_1(\bx)}=0$ which gives the smallness of $\int_{B(0;\epsilon)}\bx\rd{\mu_1(\bx)}$. Finally, $R_{\textnormal{LIC}}[W_\alpha]=0$ is a direct consequence of the construction of $\mu$, by taking the probability measures $\rho_0=\frac{\chi_{B(0;\epsilon)}}{|B(0;\epsilon)|}$ and $\rho_1=\rho_0+\theta \mu$ for sufficiently small $\theta>0$, and calculating $\frac{\rd^2}{\rd{t}^2}E [(1-t)\rho_0+t\rho_1] = 2\theta^2 E[\mu]<0$.}

As a consequence, if $0\le \alpha \le \alpha_*$, then the minimizer is unique up to translation, which was derived as a continuous function supported on some ellipse or its degenerate analog. If $\alpha$ is slightly larger than $\alpha_*$, then numerical simulation showed that a minimizer may consist of many line segments that roughly arranged inside an ellipse.

\subsection{Organization of this paper}

The rest of this paper is organized as follows: in Section \ref{sec_prep} we recall the Euler-Lagrange condition for minimizers, extend its sufficiency to certain cases with finite LIC radius, and apply it to get a uniform upper bound of minimizer size for $W_{a,b}$. In Section \ref{sec_Fou} we give a general framework for justifying the Fourier representation \eqref{EFou} for the interaction energy, which we will apply to power-law and other potentials. In particular, in Section \ref{sec_proof2ln} we use the Fourier representation to prove the LIC property of $W_{2,\ln}$, and give the explicit formula of its minimizer by verifying the Euler-Lagrange condition. In Section \ref{sec_poin} we prove the Poincar\'e-type inequality, Theorem \ref{thm_powercomp}, and its logarithmic analogue, Theorem \ref{thm_powercomplog}. Using these inequality, we prove Theorems \ref{thm_main1} and \ref{thm_main3} in Section \ref{sec_proof1}. Finally we prove Theorem \ref{thm_main2} in Section \ref{sec_proof2}.

Within this paper, a \emph{mollifier} $\phi$ refers to a smooth radially-decreasing nonnegative function supported on $\overline{B(0;1)}$ with $\int_{\mathbb{R}^d}\phi\rd{\bx}=1$, with the scaling notation $\phi_\epsilon(\bx) = \frac{1}{\epsilon^d}\phi(\frac{\bx}{\epsilon})$. A \emph{truncation function} $\Phi$ refers to a smooth radially-decreasing nonnegative function supported on $\overline{B(0;2)}$ with $\Phi=1$ on $B(0;1)$, with the scaling notation $\Phi_R(\bx) = \Phi(\frac{\bx}{R})$.

\section{Preparations}\label{sec_prep}

\subsection{The Euler-Lagrange condition}

We first recall the Euler-Lagrange condition for energy minimizers and its relation with the LIC property. We introduce a technical condition on $W$:
\begin{equation}\begin{split}
		\textnormal{{\bf (W1)}:} & \textnormal{ $W$ is continuous on $\mathbb{R}^d\backslash \{0\}$, and $W(0)=\lim_{\bx\rightarrow0}W(\bx)\in \mathbb{R}\cup\{\infty\}$; }\\
		& \textnormal{for any $R>0$, }\frac{1}{|B(\bx;\epsilon)|}\int_{B(\bx;\epsilon)}W(\by)\rd{\by} \le C_1(R) + C_2(R) W(\bx),\quad \forall |\bx|<R,\, 0<\epsilon<1\,,
\end{split}\end{equation}
which is analogous to \cite[Assumption {\bf (H-s)}]{CS21}.

\begin{remark}\label{rmk_W1}
	If $W$ has a singularity at 0, the second line in the condition {\bf (W1)} requires that the singularity is more or less regulated. For example, it is easy to verify that the power-law potential $-\frac{|\bx|^b}{b},\,-d<b\le 0$ or its logarithmic counterpart $\frac{|\bx|^b}{b}\ln|\bx|,\,-d<b< 0$ satisfies {\bf (W1)}. More generally, any potential $W$ which is continuous on $\mathbb{R}^d\backslash \{0\}$ with $-\frac{|\bx|^b}{b}\lesssim W(\bx)\lesssim -\frac{|\bx|^b}{b},\,-d<b\le 0$ near 0 (or its logarithmic counterpart) satisfies {\bf (W1)}.
\end{remark}

\begin{theorem}[\cite{BCLR13_1,CS21}, Euler-Lagrange condition for minimizer]\label{thm_EL}
	Under the same assumptions as Theorem \ref{thm_exist2}, let $\rho$ be a minimizer of $E$. Then there exists a constant $C_0$ such that
	\begin{equation}\label{thm_EL_1}
		W*\rho = C_0,\quad \rho\textnormal{-a.e. }\,,
	\end{equation}
	and
	\begin{equation}\label{thm_EL_2}
		W*\rho \ge C_0,\quad \textnormal{a.e. }\mathbb{R}^d\,.
	\end{equation}
	$C_0$ is given by $2E[\rho]$. There also holds
	\begin{equation}\label{thm_EL_3}
		W*\rho \le C_0,\quad \textnormal{on }\supp\rho\,.
	\end{equation}

	Further assume that $W$ is LIC and {\bf (W1)} holds. Then any $\rho\in\cM(\mathbb{R}^d)$ satisfying \eqref{thm_EL_1}\eqref{thm_EL_2} for some $C_0$ is the unique minimizer for $W$ up to translation.
\end{theorem}
Here $\rho\textnormal{-a.e.}$ means that the failure set for the condition has $\rho$ measure zero. 

The necessity of the Euler-Lagrange condition \eqref{thm_EL_1}\eqref{thm_EL_2}\eqref{thm_EL_3} is given by \cite[Theorem 4, Proposition 1]{BCLR13_1} on $\mathbb{R}^d$. We also mention that an earlier version appeared in \cite[Theorem 1.3]{saffbook} for a special interaction potential, but its proof indeed works for general potentials.

The sufficiency of the Euler-Lagrange condition for general LIC and {\bf (W1)} potentials is given by \cite[Theorem 2.4]{CS21}. An earlier version can be found in \cite[Theorem 1.2]{MRS} for a special interaction potential.

\subsection{LIC radius and the Euler-Lagrange condition}\label{sec_LICRER}

The sufficiency of the Euler-Lagrange condition still holds as long as $R_{\textnormal{LIC}}[W]$ is greater than an upper bound of minimizer size.

\begin{lemma}\label{lem_LICREL1}
	Under the same assumptions as Lemma \ref{lem_LICR}, assume $\rho\in \cM(\overline{B(0;R_*)})$ satisfies $\int_{\mathbb{R}^d}\bx\rd{\rho(\bx)}=0$ and
\begin{equation}\label{thm_EL_1Rst}
		W*\rho = C_0,\quad \rho\textnormal{-a.e. }\,,
	\end{equation}
	and
	\begin{equation}\label{thm_EL_2Rst}
		W*\rho \ge C_0,\quad \textnormal{a.e. }B(0;R_*+\epsilon)\,,
	\end{equation}
	for some $C_0\in\mathbb{R}$ and $\epsilon>0$. Then $\rho$ is the unique minimizer of $E$ up to translation.
\end{lemma}

\begin{proof}
	By Lemma \ref{lem_LICR}, the minimizer of $E$ is unique up to translation. Since $R_*$ is an upper bound of minimizer size, we see that there is a unique minimizer $\rho_\infty$ of $E$ satisfying  $\int_{\mathbb{R}^d}\bx\rd{\rho_\infty(\bx)}=0$, and we have $\supp\rho_\infty\subset \overline{B(0;R_*)}$.
		
	If we assume the contrary that $\rho$ is not a minimizer, then $E[\rho]>E[\rho_\infty]$. Since $R_{\textnormal{LIC}}[W]> R_*$, we see that $E$ is strictly convex along the linear interpolation between $\rho$ and $\rho_\infty$, and the same is true if $\rho_\infty$ is slightly mollified. Then one can proceed similarly as the proof of \cite[Theorem 2.4]{CS21} to get a contradiction (where one notices that \eqref{thm_EL_2Rst} is effective in the support of a slightly mollified version of $\rho_\infty$).
\end{proof}

If the condition $R_{\textnormal{LIC}}[W]> R_*$ is not known, then one can still derive Wasserstein-infinity local minimizers from the Euler-Lagrange condition, as long as one works within the radius $R_{\textnormal{LIC}}[W]$.

\begin{lemma}
	Under the same assumptions as Theorem \ref{thm_exist2}, assume {\bf (W1)} and $R:=R_{\textnormal{LIC}}[W] > 0$. Assume $\rho\in \cM(\mathbb{R}^d)$ satisfies $\supp\rho\subset\overline{B(0;R_0)}$ for some $R_0<R$, $\int_{\mathbb{R}^d}\bx\rd{\rho(\bx)}=0$ and
	\begin{equation}\label{thm_EL_1R}
		W*\rho = C_0,\quad \rho\textnormal{-a.e. }\,,
	\end{equation}
	and
	\begin{equation}\label{thm_EL_2R}
		W*\rho \ge C_0,\quad \textnormal{a.e. }B(0;R_0+\epsilon)\,,
	\end{equation}
	for some $C_0\in\mathbb{R}$ and $0<\epsilon<R-R_0$. Then $\rho$ is a Wasserstein-infinity local minimizer of $E$, in the sense that $E[\rho_1]\ge E[\rho]$ for any $\rho_1\ne \rho$ satisfying $d_\infty(\rho,\rho_1)<\frac{1}{3}\epsilon$ (where $d_\infty$ denotes the Wasserstein-infinity distance).
\end{lemma}

\begin{proof}
	First, using \eqref{thm_EL_1R} and \eqref{thm_EL_2R}, one can proceed as in the proof of Lemma \ref{lem_LICREL1} to prove that $\rho$ is the unique minimizer of $E$ within the class of probability measures $\{\rho'\in \cM(\overline{B(0;R_0+\frac{2}{3}\epsilon)}):\int_{\mathbb{R}^d}\bx\rd{\rho'(\bx)}=0\}$.
	
	Then let $\rho_1\ne \rho$ satisfy $d_\infty(\rho,\rho_1)<\frac{1}{3}\epsilon$. It implies that $\supp\rho_1\subset \overline{B(0;R_0+\frac{1}{3}\epsilon)}$ and $\bx_c:=\int_{\mathbb{R}^d}\bx\rd{\rho_1(\bx)}$ satisfies $|\bx_c|\le \frac{1}{3}\epsilon$. Therefore $\tilde{\rho}_1(\bx) := \rho_1(\bx+\bx_c)$ satisfies $E[\tilde{\rho}_1]=E[\rho_1]$, $\supp\tilde{\rho}_1\subset \overline{B(0;R_0+\frac{2}{3}\epsilon)}$ and $\int_{\mathbb{R}^d}\bx\rd{\tilde{\rho}_1(\bx)}=0$. Then it follows from the previous paragraph $E[\rho_1]=E[\tilde{\rho}_1]\ge E[\rho]$.
\end{proof}

Although not pursued in this paper, the above two lemmas allow one to obtain explicit formulas for global/local minimizers by verifying the Euler-Lagrange condition, as long as the LIC radius is sufficiently large. 

\subsection{Uniform estimate on the size of minimizers}

In this subsection we give uniform upper bounds of minimizer size for $W_{a,b}$ and $W_{b,\ln}$.

\begin{theorem}\label{thm_existab}
	Consider the power-law potentials $W_{a,b}$ in \eqref{Wab}. For fixed $b>-d$ and fixed $a_0>\max\{b,0\}$, there exists $R_*(b,a_0)$ that is an upper bound of minimizer size for any $W_{a,b}$ with $a\ge a_0$.
\end{theorem}

\begin{proof}
	$W_{a,b}$ with $-d<b<a$ satisfies the assumptions in \cite[Theorem 1.4]{CCP15}. Applying this theorem shows that there exists a minimizer of $E_{a,b}$, and any minimizer of $E_{a,b}$ is compactly supported. Then we proceed to estimate the size of a minimizer $\rho$ of $E_{a,b}$. We will only treat the case $b\ne 0$, because the case of $b=0$ (logarithmic repulsion) can be handled in a similarly way.
	
	First notice that 
	\begin{equation}\begin{split}
		E_{a,b}\Big[\frac{\chi_{B(0;1/2)}}{|B(0;1/2)|}\Big] = & \frac{1}{2a|B(0;1/2)|^2}\int_{B(0;1/2)} \int_{B(0;1/2)} |\bx-\by|^a \rd{\bx}\rd{\by} \\
		& -  \frac{1}{2b|B(0;1/2)|^2}\int_{B(0;1/2)} \int_{B(0;1/2)} |\bx-\by|^b \rd{\bx}\rd{\by} \\
		\le & {\frac{1}{2a} + C_1 \le \frac{1}{2a_0} + C_1 =: C_2\,,}
	\end{split}\end{equation}
	where $C_1$ denotes the integral term involving $b$, which  does not depends on $a$. As a consequence, $C_2$ depends on $b$ and $a_0$, independent of $a$, and 
	\begin{equation}
		E_{a,b}[\rho] \le E_{a,b}\Big[\frac{\chi_{B(0;1/2)}}{|B(0;1/2)|}\Big] \le C_2\,,
	\end{equation}
	since $\rho$ is a minimizer of $E_{a,b}$.
	
	Then we claim that there exists $r_*\ge 2$, depending on $b$ and $a_0$, such that
	\begin{equation}\label{claim_t}
		\min_{t\ge r} (\frac{t^a}{a}-\frac{t^b}{b}) \ge \frac{r^a}{2a},\quad \forall a\ge a_0,\,r\ge r_*\,.
	\end{equation}
	In fact, if $b<0$ then \eqref{claim_t} is trivial. If $b>0$, then for any $a\ge a_0$ and $t\ge r_*$ we estimate
	\begin{equation}
		\frac{t^b}{b}\cdot \frac{a}{t^a} = \frac{a}{b}\cdot t^{b-a} \le \frac{a}{b}\cdot r_*^{b-a} \le \frac{1}{b} \cdot 2^{(b-a)/2}a \cdot r_*^{(b-a_0)/2}\,.
	\end{equation}
	Notice that $\lim_{a\rightarrow\infty}2^{(b-a)/2}a = 0$, and thus \begin{equation}
		C_3:=\sup_{a\ge a_0}2^{(b-a)/2}a<\infty\,,
	\end{equation}
	is independent of $a$. Then we choose
	\begin{equation}
		r_* = \Big(\frac{b}{2C_3}\Big)^{2/(b-a_0)}\,,
	\end{equation}
	so that $\frac{t^b}{b}\cdot \frac{a}{t^a}\le \frac{1}{2}$ for any $a\ge a_0$ and $t\ge r_*$, and \eqref{claim_t} follows. 
		
	{
	Then, similar to \cite[Lemma 2.6]{CCP15}, we take 
	\begin{equation}\label{R1}
		R_1 = \max\left\{r_*,\Big(\frac{2}{C_4}\big(2C_2-\min\big\{-\frac{1}{b},0\big\}+1\big)\Big)^{2/a_0}\right\}\,,
	\end{equation}
	where
	\begin{equation}\label{C5}
		C_4 = \min_{a>0}\frac{2^{a/2}}{2a} > 0\,.
	\end{equation}
	Then $R_1$ is independent of $a$ and satisfies $R_1\ge r_* \ge 2$. We claim that
	\begin{equation}\label{claim_r}
		\rho(B(\bx_0;R_1)) \ge \frac{1}{2},\quad \forall \bx_0\in\supp\rho\,.
	\end{equation}
	Assume the contrary, i.e., $\rho(B(\bx_0;R_1)^c) > \frac{1}{2}$ for some $\bx_0\in\supp\rho$, and estimate
	\begin{equation}\begin{split}
		(W_{a,b}*\rho)(\bx_0) = & \int_{B(\bx_0;R_1)} W_{a,b}(\bx_0-\by)\rd{\rho(\by)} + \int_{B(\bx_0;R_1)^c} W_{a,b}(\bx_0-\by)\rd{\rho(\by)}  \\
		\ge & \min W_{a,b}\int_{B(\bx_0;R_1)} \rd{\rho(\by)} + \min_{t\ge R_1} \big(\frac{t^a}{a}-\frac{t^b}{b}\big)\int_{B(\bx_0;R_1)^c} \rd{\rho(\by)}  \\
		\ge & \min\big\{\frac{1}{a}-\frac{1}{b},0\big\} + \frac{R_1^a}{4a} \\
		\ge & \min\big\{-\frac{1}{b},0\big\} + \frac{2^{a/2}}{4a}\cdot R_1^{a_0/2} \\
		\ge & \min\big\{-\frac{1}{b},0\big\} + \frac{1}{2}C_4\cdot R_1^{a_0/2}\,, \\
	\end{split}\end{equation}
	by using \eqref{claim_t}. Then, by \eqref{R1}, we see that
	\begin{equation}
		(W_{a,b}*\rho)(\bx_0) > 2C_2 \ge 2E_{a,b}[\rho]\,,
	\end{equation}
	which violates the Euler-Lagrange condition \eqref{thm_EL_3}. Therefore \eqref{claim_r} is proved.
}

	
	Finally we will derive the conclusion.  First, we notice the estimate
	\begin{equation}
		W_{a,b}(\bx) \ge \frac{R_1^a}{2a} \ge \frac{2^{a/2}}{2a} \cdot R_1^{a/2} \ge C_4 R_1^{a_0/2} \ge 2\Big(2C_2-\min\big\{-\frac{1}{b},0\big\}+1\Big) ,\quad \forall |\bx|\ge R_1\,,
	\end{equation}
	due to \eqref{claim_t}. Take a point $\bx_0\in\supp\rho$. Then for any $\bx$ with $|\bx-\bx_0|>2R_1$, since the ball $B(\bx_0;R_1)$ is contained in $B(\bx;R_1)^c$, we have
	\begin{equation}\begin{split}
			(W_{a,b}*\rho)(\bx) 
			= & \int_{B(\bx;R_1)} W_{a,b}(\bx-\by)\rd{\rho(\by)} + \int_{B(\bx;R_1)^c} W_{a,b}(\bx-\by)\rd{\rho(\by)} \\
			\ge & \min\big\{-\frac{1}{b},0\big\} + \frac{1}{2}\cdot2\Big(2C_2-\min\big\{-\frac{1}{b},0\big\}+1\Big) > 2C_2 \ge 2E[\rho]\,,
	\end{split}\end{equation}
	due to \eqref{claim_r}. This implies $\supp\rho\subset \overline{B(\bx_0;2R_1)}$ because otherwise the Euler-Lagrange condition \eqref{thm_EL_3} would be violated. 
	This implies that the center of mass $\int_{\mathbb{R}^d}\bx\rd{\rho(\bx)}$ is also in $\overline{B(\bx_0;2R_1)}$, and thus obtain an upper bound of minimizer size 
	\begin{equation}
		R_*(b,a_0)= 4R_1\,.
	\end{equation}
	
\end{proof}

\begin{remark}
	We believe that a stronger uniform estimate for the minimizer size for $W_{a,b}$ is possible, say, in both parameters $(a,b)$, but one needs to study whether the minimizer size stays $O(1)$ when $a$ or $b$ is close to 0 or $-d$, or $a$ is close to $b$. Also, when $a<0$ one has $\lim_{|\bx|\rightarrow\infty}W_{a,b}(\bx)=0$, and one needs to generalize the combinatoric argument in the proof of \cite[Lemma 2.9]{CCP15}.  We leave this as future work.
\end{remark}

\begin{theorem}\label{thm_existablog}
	Consider the logarithmic power-law potentials $W_{b,\ln}$ in \eqref{Wbln}. For any $b_0>0$, there exists $R_{*,\ln}(b_0)$ that is an upper bound of minimizer size for any $W_{b,\ln}$ with $b\ge b_0$.
\end{theorem}

\begin{proof}
	Since the proof is analogous to that of Theorem \ref{thm_existab}, we only give an outline of the most important steps. Let $\rho$ be a minimizer of $E_{b,\ln}$ (the energy associated to $W_{b,\ln}$). First notice that
	\begin{equation}\begin{split}
		E_{b,\ln}\Big[\frac{\chi_{B(0;1/2)}}{|B(0;1/2)|}\Big] \le 0\,,
	\end{split}\end{equation}
	since $W_{b,\ln}(\bx) = \frac{|\bx|^b}{b}\ln|\bx| \le 0$ for any $|\bx|\le 1$ if $b>0$. It follows that $E_{b,\ln}[\rho]\le 0$.
	
	Then notice that 
	\begin{equation}\label{claim_tlog}
		\min_{t\ge r} \frac{t^b}{b}\ln t \ge \frac{r^b}{b},\quad \forall b>0,\,t\ge e\,.
	\end{equation}
	{
	Then we take
	\begin{equation}
		R_1 = 	\max\left\{e,\Big(\frac{1}{C_4}\big(\frac{e^{-1}}{b_0^2}+1\big)\Big)^{2/b_0}\right\}\,,
	\end{equation}
	where $C_4$ is defined in \eqref{C5}, and proceed to prove  \eqref{claim_r}. We again assume the contrary and estimate
	\begin{equation}\begin{split}
		(W_{b,\ln}*\rho)(\bx_0) = & \int_{B(\bx_0;R_1)} W_{b,\ln}(\bx_0-\by)\rd{\rho(\by)} + \int_{B(\bx_0;R_1)^c} W_{b,\ln}(\bx_0-\by)\rd{\rho(\by)}  \\
		\ge & \min W_{b,\ln}\int_{B(\bx_0;R_1)} \rd{\rho(\by)} + \min_{t\ge R_1} \frac{t^b}{b}\ln t\int_{B(\bx_0;R_1)^c} \rd{\rho(\by)}  \\
		\ge & -\frac{e^{-1}}{b^2}  + \frac{R_1^b}{2b} 
		\ge  -\frac{e^{-1}}{b_0^2} + \frac{2^{b/2}}{2b}\cdot R_1^{b_0/2} 
		\ge  -\frac{e^{-1}}{b_0^2} +C_4\cdot R_1^{b_0/2} \\
		> & 0 \ge 2E_{b,\ln}[\rho]\,. \\
	\end{split}\end{equation}
	 Therefore we  get a contradiction against the Euler-Lagrange condition, which proves \eqref{claim_r}.
}
	
	
	Finally we notice that 
	\begin{equation}
		W_{b,\ln}(\bx) \ge \frac{R_1^b}{b} \ge 2C_4 R_1^{b_0/2} \ge 2\Big(\frac{e^{-1}}{b_0^2}+1\Big) ,\quad \forall |\bx|\ge R_1\,,
	\end{equation}
	by \eqref{claim_tlog}. This leads to $\supp\rho\subset \overline{B(\bx_0;2R_1)}$ and thus an upper bound of minimizer size $R_{*,\ln}(b_0)= 4R_1$ similar as the proof of Theorem \ref{thm_existab}.

\end{proof}

\section{Fourier representation of the interaction energy}\label{sec_Fou}

In this section we give a general theory for the Fourier representation of the interaction energy. For certain interaction potentials $W$ and compactly-supported signed measures $\mu$, we will justify the formal relation
\begin{equation}\label{EFou}
	E[\mu] = \frac{1}{2}\int_{\mathbb{R}^d\backslash \{0\}} \hat{W}(\xi)|\hat{\mu}(\xi)|^2\rd{\xi}\,.
\end{equation}

We introduce the following assumption on $W$ which is slightly weaker than {\bf (W0)}.
\begin{equation}\begin{split}
		\text{{\bf (W0-w)}: $W$} & \text{ is locally integrable, bounded from below on compact sets} \\ & \text{and lower-semicontinuous, with $W(\bx)=W(-\bx)$.}
\end{split}\end{equation}

\begin{definition}
	Let $W$ satisfy {\bf (W0-w)} and $k$ be a nonnegative integer. We say $W$ is \emph{Fourier representable at level $k$} if $\hat{W}$ is a locally integrable function on $\mathbb{R}^d\backslash\{0\}$ and \eqref{EFou} holds for any compactly-supported signed measure $\mu$ satisfying $E[|\mu|]<\infty$ and $\int_{\mathbb{R}^d}\bx^{\otimes j}\rd{\mu(\bx)} = 0,\,j=0,1,\dots,k-1$.
\end{definition}

For the purpose of analyzing the LIC radius of $W$, Fourier representable property at level 2 is crucial if one wants to apply \eqref{EFou}, see Theorem \ref{thm_RLICab}. Fourier representable property at level 1 would be necessary for analyzing the uniqueness of minimizer in the presence of external potentials or domain constraints due to the lack of translation invariance.

We will first give sufficient conditions to guarantee the Fourier representable property at levels 0, 1, 2. Then we apply them to power-law potentials to obtain their Fourier representable property in Theorem \ref{thm_Fourep}. Finally we use the Fourier representation to give an equivalent formulation of the LIC radius in Theorem \ref{thm_RLICab}, and prove Theorem \ref{thm_Wbln}.

\subsection{Fourier representable potentials at level 0}

We first prove the convergence of the energy for measures via mollification.

\begin{lemma}\label{lem_W1}
	Assume {\bf (W0-w)}, {\bf (W1)}. Let $\mu$ be a compactly supported signed measure on $\mathbb{R}^d$, and $\phi_\epsilon$ be a mollifier (see the last paragraph of Section \ref{sec_intro}). If $E[|\mu|]<\infty$, then
	\begin{equation}
		\lim_{\epsilon\rightarrow 0} E [\mu*\phi_\epsilon] = E [\mu]\,.
	\end{equation}
\end{lemma}
A variant of this lemma appeared in \cite[Lemma 2.5]{CS21}.
\begin{proof}
	We rewrite as
	\begin{equation}
		2E [\mu*\phi_\epsilon] = \int_{\mathbb{R}^d} (W*(\mu*\phi_\epsilon))(\mu*\phi_\epsilon)\rd{\bx} = \int_{\mathbb{R}^d} (W*\psi_\epsilon*\mu))\mu\rd{\bx} 
	\end{equation}
	where $\psi_\epsilon:=\phi_\epsilon*\phi_\epsilon$ is another mollifier with $\psi=\phi*\phi$ supported on $\overline{B(0;2)}$. Since $\mu$ is compactly supported, we may use {\bf (W1)} with {$R=2\diam(\supp\mu)$} to see that the last integrand is bounded by $((C_1+C_2 W)*|\mu|)|\mu|$ whose integral is finite by the assumption $E [|\mu|]<\infty$. Therefore, applying the dominated convergence theorem, we get the conclusion since {\bf (W1)} implies that $\lim_{\epsilon\rightarrow 0}(W*\psi_\epsilon)(\bx) = W(\bx)\in \mathbb{R}\cup\{\infty\}$ for any $\bx\in\mathbb{R}^d$.
\end{proof}

\begin{lemma}\label{lem_W2}
	Assume {\bf (W0-w)}, {\bf (W1)}. Let $\rho\in\cM(\mathbb{R}^d)$ be compactly supported, and $\phi_\epsilon$ be a mollifier. If $E [\rho]=\infty$, then
	\begin{equation}
		\lim_{\epsilon\rightarrow 0} E [\rho*\phi_\epsilon] = \infty\,.
	\end{equation}
\end{lemma}

\begin{proof}
	By adding constant and modifying $W$ at long distance, we may assume $\inf W=0$ without loss of generality. 
	
	There must hold $W(0)=\infty$ because otherwise $W$ would be continuous on $\mathbb{R}^d$ due to {\bf (W1)}, and $E [\rho]=\infty$ cannot hold. If 
	\begin{equation}
		\iint_{\bx=\by} W(\bx-\by)\rd{(\rho(\bx)\otimes\rho(\by))} = \infty\,,
	\end{equation}
	then $\rho$ contains a Dirac mass, and the conclusion follows easily from $\lim_{\by\rightarrow0}W(\by)=W(0)=\infty$. Otherwise one must have
	\begin{equation}
		\iint_{\bx=\by} W(\bx-\by)\rd{(\rho(\bx)\otimes\rho(\by))} = 0\,,
	\end{equation}
	and then
	\begin{equation}
		\lim_{\alpha\rightarrow 0^+}\iint_{|\bx-\by|>\alpha} W(\bx-\by)\rd{(\rho(\bx)\otimes\rho(\by))} = 2E [\rho] = \infty\,.
	\end{equation}
	Let $M>0$ be given. Then, we first take $\alpha>0$ such that $\iint_{|\bx-\by|>\alpha} W(\bx-\by)\rd{(\rho(\bx)\otimes\rho(\by))} \ge M+1$. Then, using the continuity of $W$ on $\mathbb{R}^d\backslash \{0\}$, we have
	\begin{equation}
		(W*\psi_\epsilon)(\bx) \ge W(\bx)-1,\quad \forall \alpha \le |\bx| \le \diam(\supp\rho)\,,
	\end{equation}
	for sufficiently small $\epsilon>0$, where $\psi_\epsilon$ is the same as in the proof of Lemma \ref{lem_W1}. This implies
	\begin{equation}
		2E [\rho*\phi_\epsilon] = \int_{\mathbb{R}^d} (W*\psi_\epsilon*\rho))\rho\rd{\bx} \ge \iint_{|\bx-\by|>\alpha} (W(\bx-\by)-1)\rd{(\rho(\bx)\otimes\rho(\by))} \ge M\,.
	\end{equation}
	This leads to the conclusion.
\end{proof}

\begin{theorem}\label{thm_convFou1}
	Assume $W$ is a tempered distribution satisfying {\bf (W0-w)}{\bf (W1)}, and $\hat{W}$ is a locally integrable function. {Assume either $\hat{W}\in L^1(\mathbb{R}^d)$, or $\hat{W}$ is one-signed on $B(0;R_0)^c$ for some $R_0>0$.} Then $W$ is Fourier representable at level 0.
\end{theorem}

We first give a lemma on the tail behavior of the convolution of two functions.

\begin{lemma}\label{lem_psi}
	Let $\psi_1$ be a continuous function on $\mathbb{R}^d$ such that $|\psi_1(\bx)|\le C_N (1+|\bx|)^{-N}$ for arbitrarily large $N$. Denote $\psi_{1,R}(\bx) = R^d \psi_1(R \bx)$ for any $R\ge 1$. Let $\psi_2$ be a locally integrable function on $\mathbb{R}^d$ such that 
	\begin{equation}\label{lem_psi_0}
		|\psi_2(\bx)| \lesssim |\bx|^M,\quad\forall |\bx|\ge 1\,,
	\end{equation}
	for some $M\in\mathbb{R}$. Then
	\begin{equation}\label{lem_psi_1}
		|(\psi_{1,R}*\psi_2)(\bx)| \lesssim  |\bx|^M,\quad \forall |\bx|>3,\quad \forall R\ge 1\,.
	\end{equation}
	
	If one further assumes that $\psi_2$ is continuous, then
	\begin{equation}\label{lem_psi_2}
		\lim_{R\rightarrow\infty}(\psi_{1,R}*\psi_2)(\bx) = \psi_2(\bx)\int_{\mathbb{R}^d}\psi_1(\by)\rd{\by},\quad \forall \bx\in\mathbb{R}^d\,.
	\end{equation}
\end{lemma}

\begin{proof}
	We first prove \eqref{lem_psi_1}. For $|\bx|>3$, we write
	\begin{equation}
		(\psi_{1,R}*\psi_2)(\bx) =  \int_{\mathbb{R}^d} R^d\psi_1(R\by)\psi_2(\bx-\by)\rd{\by} 
		=  \int_{|\by|< |\bx|/2} + \int_{|\bx|/2 \le |\by| < 2|\bx|} + \int_{|\by|\ge 2|\bx|} \,,\\
	\end{equation}
	and then estimate the three parts separately.
	
	When $|\by|< |\bx|/2$, we have $|\bx|/2 \le |\bx-\by| \le 2|\bx| $ where $ |\bx|/2 > 1$, and thus $|\psi_2(\bx-\by)| \lesssim |\bx|^M$. Therefore
	\begin{equation}\begin{split}
		\left|\int_{|\by|< |\bx|/2}R^d\psi_1(R\by)\psi_2(\bx-\by)\rd{\by} \right| \lesssim & |\bx|^M \int_{|\by|< |\bx|/2}R^d|\psi_1(R\by)|\rd{\by} \\
		= & |\bx|^M \int_{|\by|< R|\bx|/2}|\psi_1(\by)|\rd{\by} \lesssim |\bx|^M\,,
	\end{split}\end{equation}
	since $\psi_1 \in L^1(\mathbb{R}^d)$.
	
	When $|\bx|/2 \le |\by| < 2|\bx|$, we have $|\psi_1(R\by)| \le C_N (1+R|\bx|)^{-N}$, and thus obtain
	\begin{equation}\begin{split}
			\left|\int_{|\bx|/2 \le |\by| < 2|\bx|}  R^d\psi_1(R\by)\psi_2(\bx-\by)\rd{\by} \right| \lesssim & C_N R^d(1+R|\bx|)^{-N}\int_{|\bx|/2 \le |\by| < 2|\bx|} |\psi_2(\bx-\by)|\rd{\by}\,. \\
	\end{split}\end{equation}
	First notice that $R^d \le (1+R|\bx|)^d$ since $|\bx| > 3$. Then we estimate
	\begin{equation}\begin{split}
			\int_{|\bx|/2 \le |\by| < 2|\bx|} & |\psi_2(\bx-\by)|\rd{\by} \le  \int_{B(0;3|\bx|)} |\psi_2(\bz)|\rd{\bz} \\
			\le & \left( \int_{B(0;1)} |\psi_2(\bz)|\rd{\bz} + \int_{1\le |\bz| \le 3|\bx|}C|\bz|^M\rd{\bz} \right) 
			\lesssim  1+|\bx|^{M+d} \,,
	\end{split}\end{equation}
	(the last quantity becomes $1+\ln|\bx|$, if $M=-d$, same for the inequality below). By taking $N=d+\max\{d+1,-M\}$ and using $R\ge 1$, we conclude that
	\begin{equation}\begin{split}
		\left|\int_{|\bx|/2 \le |\by| < 2|\bx|}\psi_1(\by)\psi_2(\bx-\by)\rd{\by} \right| \lesssim & (1+R|\bx|)^{\min\{-d-1,M\}}  + (1+R|\bx|)^{\min\{-d-1,M\}} |\bx|^{M+d} \\
		\le & (1+|\bx|)^{M}  + (1+|\bx|)^{-d-1} |\bx|^{M+d} \\
		\lesssim & |\bx|^M + |\bx|^{M-1}\,.
	\end{split}\end{equation}	
	Here the extra $-1$ on the last exponent can help absorb the extra $\ln|\bx|$ factor in case it appears.
	
	When $|\by|\ge 2|\bx|$, we have $|\by|/2 \le |\bx-\by| \le 2|\by|$ where $|\by|/2>1$, and thus $|\psi_2(\bx-\by)| \lesssim |\by|^M$. Combining with $|\psi_1(R\by)| \le C_N (1+R|\by|)^{-N}$, $R^d \le (1+R|\by|)^d$ and choosing $N=d + \max\{M+d+1,d\}$, we obtain
	\begin{equation}\begin{split}
			\left|\int_{|\by|\ge 2|\bx|}  R^d\psi_1(R\by)\psi_2(\bx-\by)\rd{\by} \right| \lesssim & \int_{|\by|\ge 2|\bx|} R^d(1+R|\by|)^{-d - \max\{M+d+1,d\}}|\by|^M \rd{\by} \\
			\le & \int_{|\by|\ge 2|\bx|} (1+R|\by|)^{\min\{-M-d-1,-d\}}|\by|^M \rd{\by} \\
			\le & \int_{|\by|\ge 2|\bx|} |\by|^{\min\{-d-1,M-d\}} \rd{\by} 
			\lesssim  |\bx|^{M}\,. \\
	\end{split}\end{equation}
	
	Combining the above estimates, \eqref{lem_psi_1} follows.
	
	The proof of \eqref{lem_psi_2} for continuous $\psi_2$ is a standard `approximation of identity' argument. Indeed, for a fixed $\bx$, write
	\begin{equation}\begin{split}
			(\psi_{1,R} & *\psi_2)(\bx) -  \psi_2(\bx)\int_{\mathbb{R}^d}\psi_1(\by)\rd{\by} \\
			= & \int_{\mathbb{R}^d} R^d\psi_1(R\by)\psi_2(\bx-\by)\rd{\by} - \psi_2(\bx)\int_{\mathbb{R}^d}R^d\psi_1(R\by)\rd{\by} \\
			= & \int_{|\by|< 1/\sqrt{R}}R^d\psi_1(R\by)(\psi_2(\bx-\by)-\psi_2(\bx)) \rd{\by} \\
			& + \int_{|\by|\ge 1/\sqrt{R}}R^d\psi_1(R\by)(\psi_2(\bx-\by)-\psi_2(\bx)) \rd{\by}\,. \\
	\end{split}\end{equation}
	We first estimate
	\begin{equation}\begin{split}
			\left| \int_{|\by|< 1/\sqrt{R}}R^d\psi_1(R\by)(\psi_2(\bx-\by)-\psi_2(\bx)) \rd{\by}   \right| 
			\le  \|\psi_1\|_{L^1}\sup_{\bz\in B(\bx;1/\sqrt{R})} |\psi_2(\bz)-\psi_2(\bx)|\,,
	\end{split}\end{equation}
	which goes to zero as $R\rightarrow\infty$, by the continuity of $\psi_2$. Then notice that \eqref{lem_psi_0} and the continuity of $\psi_2$ imply that $\psi_2(\bz) \lesssim (1+|\bz|)^M$ for any $\bz\in\mathbb{R}^d$. Recalling that $\bx$ is fixed, this allows us to estimate
	\begin{equation}\begin{split}
			& \left| \int_{|\by|\ge 1/\sqrt{R}}R^d  \psi_1(R\by)(\psi_2(\bx-\by)-\psi_2(\bx)) \rd{\by}   \right| \\
			\lesssim & \int_{|\by|\ge 1/\sqrt{R}}R^d|\psi_1(R\by)| (1+|\by|)^M \rd{\by} \\
			\lesssim & C_N\int_{|\by|\ge 1/\sqrt{R}}R^d(1+R|\by|)^{-N} (1+|\by|)^M \rd{\by} \\
			= & C_N\int_{|\by|\ge \sqrt{R}}(1+|\by|)^{-N} (1+|\frac{\by}{R}|)^M \rd{\by} \,.\\
	\end{split}\end{equation}
	Notice that $|\frac{\by}{R}| \le |\by|$ for any $R\ge 1$, and thus $(1+|\frac{\by}{R}|)^M \le (1+|\by|)^{\max\{M,0\}}$. Therefore, by choosing $N = d+1+\max\{M,0\}$, we obtain
	\begin{equation}\begin{split}
			& \left| \int_{|\by|\ge 1/\sqrt{R}}R^d  \psi_1(R\by)(\psi_2(\bx-\by)-\psi_2(\bx)) \rd{\by}   \right| 
			\lesssim  \int_{|\by|\ge \sqrt{R}}(1+|\by|)^{-d-1}  \rd{\by} \lesssim R^{-1/2}\,, \\
	\end{split}\end{equation}
	which also goes to zero as $R\rightarrow\infty$. Therefore \eqref{lem_psi_2} follows.
\end{proof}

\color{black}
\begin{proof}[Proof of Theorem \ref{thm_convFou1}]
	Since $W$ is a tempered distribution, $\hat{W}$ is well-defined as a tempered distribution. Let $\mu$ be a compactly supported signed measure with $E[|\mu|]<\infty$. We take a mollifier $\phi_\epsilon$ and a truncation function $\Phi_R$ (see the last paragraph of Section \ref{sec_intro}). Define
	\begin{equation}\label{fepsR1}
		f_{\epsilon,R} = \int_{\mathbb{R}^d} ((W\Phi_R)*(\mu*\phi_\epsilon))(\mu*\phi_\epsilon)\rd{\bx}\,.
	\end{equation}
	Notice that both $(W\Phi_R)*(\mu*\phi_\epsilon)$ and $\mu*\phi_\epsilon$ are $L^2$ functions. Therefore we may apply the Plancherel identity to get
	\begin{equation}\label{fepsR2}\begin{split}
		f_{\epsilon,R} = & \int_{\mathbb{R}^d} \cF((W\Phi_R)*(\mu*\phi_\epsilon))\cdot \overline{\cF(\mu*\phi_\epsilon)}\rd{\xi} \\
		= & \int_{\mathbb{R}^d}(\hat{W}*\hat{\Phi}_R)|\hat{\mu}|^2|\hat{\phi}_\epsilon|^2\rd{\xi} = \int_{\mathbb{R}^d}\hat{W}\big((|\hat{\mu}|^2|\hat{\phi}_\epsilon|^2) *\hat{\Phi}_R\big)\rd{\xi}\,.
	\end{split}\end{equation}
	We first fix $\epsilon$, and notice that $f_{\epsilon,R}$ is constant in $R$ for sufficiently large $R$, and equal to
	\begin{equation}\label{feps1}
		f_{\epsilon} = \int_{\mathbb{R}^d} (W*(\mu*\phi_\epsilon))(\mu*\phi_\epsilon)\rd{\bx}\,,
	\end{equation}
	since $\mu*\phi_\epsilon$ is compactly supported. Then, in \eqref{fepsR2} we send $R\rightarrow\infty$. By assumption $\hat{W}$ is a locally integrable function which is also a tempered distribution, and thus 
	
	\begin{equation}\label{hatWb}
		\int_{B(0;1)^c}|\hat{W}(\xi)|\cdot|\xi|^{-k}\rd{\bx}<\infty\,,
	\end{equation}
	for sufficiently large $k$. 
	
	Now we treat the convolution on the RHS of \eqref{fepsR2}. Notice that $|\hat{\mu}|^2|\hat{\phi}_\epsilon|^2$ is a Schwartz function, and $\hat{\Phi}_R(\xi) = R^d \hat{\Phi}(R\xi)$, with $\hat{\Phi}$ a Schwartz function. Applying \eqref{lem_psi_1} in Lemma \ref{lem_psi} with $\psi_1 = \hat{\Phi}$ and $\psi_2 = |\hat{\mu}|^2|\hat{\phi}_\epsilon|^2$ with $M=-k$ ($k$ as given in \eqref{hatWb}), we see that 
	\begin{equation}
		\big((|\hat{\mu}|^2|\hat{\phi}_\epsilon|^2) *\hat{\Phi}_R\big)(\xi) \lesssim (1+|\xi|)^{-k},\quad \forall |\xi|>3,\quad \forall R\ge 1\,.
	\end{equation}
	This inequality also holds for $|\xi|\le 3$ by Young's inequality, since $ |\hat{\mu}|^2|\hat{\phi}_\epsilon|^2\in L^\infty$ and $\|\hat{\Phi}_R\|_{L^1}=\|\hat{\Phi}\|_{L^1}<\infty$. Therefore we see that the last integrand $\hat{W}\big((|\hat{\mu}|^2|\hat{\phi}_\epsilon|^2) *\hat{\Phi}_R\big)$ of \eqref{fepsR2} is dominated by an integrable function $C|\hat{W}(\xi)|(1+|\xi|)^{-k}$ for any $R\ge 1$. 
	
	Applying \eqref{lem_psi_2} in Lemma \ref{lem_psi}, combining with $\int_{\mathbb{R}^d} \hat{\Phi}\rd{\xi}= \Phi(0)=1$, gives $\lim_{R\rightarrow\infty}\big((|\hat{\mu}|^2|\hat{\phi}_\epsilon|^2) *\hat{\Phi}_R\big)(\xi) = |\hat{\mu}(\xi)|^2|\hat{\phi}_\epsilon(\xi)|^2$ for every $\xi$. Therefore, applying the dominated convergence theorem to \eqref{fepsR2}, we see that
	\begin{equation}\label{feps2}
		f_{\epsilon} = \int_{\mathbb{R}^d}\hat{W}|\hat{\mu}|^2|\hat{\phi}_\epsilon|^2\rd{\xi}\,.
	\end{equation}
	
	\color{black}
	
	Then we send $\epsilon\rightarrow 0$. In \eqref{feps1}, since $W$ satisfies {\bf (W1)} and $E[|\mu|]<\infty$, we apply Lemma \ref{lem_W1} to get
	\begin{equation}
		\lim_{\epsilon\rightarrow 0}f_{\epsilon} = \int_{\mathbb{R}^d} (W*\mu)\mu\rd{\bx} \,.
	\end{equation}

	In \eqref{feps2}, notice that $|\hat{\mu}|^2\in L^\infty$, and $\hat{\phi}_\epsilon=\hat{\phi}(\epsilon\cdot)$ is uniformly bounded and converges uniformly to 1 on compact sets. Therefore, if $\hat{W}\in L^1(\mathbb{R}^d)$, then the integrand in \eqref{feps2} is dominated by an integrable function $C|\hat{W}(\xi)|$ for any $\epsilon>0$. Therefore, by the dominated convergence theorem, we get
	\begin{equation}\label{phiepsconv}
		\lim_{\epsilon\rightarrow 0}\int_{\mathbb{R}^d}\hat{W}|\hat{\mu}|^2|\hat{\phi}_\epsilon|^2\rd{\xi} = \int_{\mathbb{R}^d}\hat{W}|\hat{\mu}|^2\rd{\xi} \,,
	\end{equation}
	 and thus \eqref{EFou} is proved.
	
	If we assume that $\hat{W}$ is one-signed on $B(0;R_0)^c$ for some $R_0>0$, then we may still apply  the dominated convergence theorem on  $B(0;R_0)$ to get
	\begin{equation}
		\lim_{\epsilon\rightarrow 0}\int_{B(0;R_0)}\hat{W}|\hat{\mu}|^2|\hat{\phi}_\epsilon|^2\rd{\xi} = \int_{B(0;R_0)}\hat{W}|\hat{\mu}|^2\rd{\xi} \,.
	\end{equation}
	Next, using the assumption that $\hat{W}$ is one-signed on $B(0;R_0)^c$, we may obtain that
	\begin{equation}\label{R0converge}
		\lim_{\epsilon\rightarrow 	0}\int_{B(0;R_0)^c}\hat{W}|\hat{\mu}|^2|\hat{\phi}_\epsilon|^2\rd{\xi} = \int_{B(0;R_0)^c}\hat{W}|\hat{\mu}|^2\rd{\xi} \,.
	\end{equation}
	To see this, we first consider the case $\hat{W}\ge 0$ on $B(0;R_0)^c$. If  $\int_{B(0;R_0)^c}\hat{W}|\hat{\mu}|^2\rd{\xi}=\infty$, then one clearly has $\lim_{\epsilon\rightarrow 	0}\int_{B(0;R_0)^c}\hat{W}|\hat{\mu}|^2|\hat{\phi}_\epsilon|^2\rd{\xi}=\infty$ since $\hat{W}|\hat{\mu}|^2|\hat{\phi}_\epsilon|^2\ge 0$ and $\hat{\phi}_\epsilon$ converges uniformly to 1 on compact sets. If $\int_{B(0;R_0)^c}\hat{W}|\hat{\mu}|^2\rd{\xi}<\infty$ then \eqref{R0converge} follows from the dominated convergence theorem. The case that $\hat{W}\le 0$ on $B(0;R_0)^c$ can be treated similarly. Therefore we again obtain \eqref{phiepsconv} and thus \eqref{EFou} is proved.
\color{black}

\end{proof}

\subsection{Fourier representable potentials at levels 1 and 2}

Then we use the Laplacian operator to study the Fourier representable property at higher levels, similar to \cite[Theorem 5.3]{CS21}.

\begin{theorem}\label{thm_convFou2}
	Assume $W$ is a tempered distribution satisfying {\bf (W0-w)}{\bf (W1)}, and $-\Delta W$ satisfies the assumptions of Theorem \ref{thm_convFou1}. Assume 
	\begin{equation}\label{thm_convFou2_1}
		|\nabla W(\bx)|\lesssim |\bx|^{b-1},\quad |\Delta W(\bx)|\lesssim |\bx|^{b-2}\,,
	\end{equation}
	for some $b < 2$ and large $|\bx|$. Then $W$ is Fourier representable at level 1.
\end{theorem}

{
We first give two lemmas on the convolution of a function with a compactly supported measure with vanishing moments.
\begin{lemma}\label{lem_convvan}
	Let $\mu$ be a compactly supported signed measure with $\int_{\mathbb{R}^d}\rd{\mu}=0$, and $f$ be a $C^1$ function on $\mathbb{R}^d\backslash \{0\}$ satisfying 
	\begin{equation}\label{lem_convvan_1}
		|\nabla f(\bx)| \lesssim |\bx|^p\,,
	\end{equation}
	for sufficiently large $|\bx|$, for some $p\in \mathbb{R}$. Then
	\begin{equation}
		|(f*\mu)(\bx)| \lesssim |\bx|^p\,,
	\end{equation}
	for sufficiently large $|\bx|$.
\end{lemma}
\begin{proof}
	Let $R$ be sufficiently large so that $\supp\mu\subset B(0;R)$, and \eqref{lem_convvan_1} is valid for $|\bx|\ge R$. For any $\bx$ with $|\bx|\ge 2R$, we write
	\begin{equation}
		(f*\mu)(\bx) = \int_{\supp\mu} f(\bx-\by)\rd{\mu(\by)} = \int_{\supp\mu} (f(\bx-\by)-f(\bx))\rd{\mu(\by)}\,.
	\end{equation}
	For each $\by\in\supp\mu$, applying the mean value theorem to $t\mapsto f(\bx-t\by)$ on $t\in [0,1]$ gives
	\begin{equation}
		f(\bx-\by)-f(\bx) = -\nabla f(\bx-t_\by\by) \cdot \by \,,
	\end{equation}
	for some $t_\by\in (0,1)$ (depending on $\by$). Since $|\bx|\ge 2R$ and $|\by|\le R$, we have $\frac{|\bx|}{2}\le |\bx-t_\by\by|\le 2|\bx|$, which implies  $|\nabla f(\bx-t_\by\by)| \lesssim |\bx|^p$. Therefore we conclude that
	\begin{equation}
		|(f*\mu)(\bx)| \lesssim |\bx|^p \int_{\supp\mu} |\by| \rd{|\mu|(\by)} \lesssim |\bx|^p\,.
	\end{equation}
\end{proof}
\begin{lemma}\label{lem_convvan2}
	Let $\mu$ be a compactly supported signed measure with $\int_{\mathbb{R}^d}\rd{\mu}=\int_{\mathbb{R}^d}\bx\rd{\mu(\bx)}=0$, and $f$ be a $C^2$ function on $\mathbb{R}^d\backslash \{0\}$ satisfying 
	\begin{equation}\label{lem_convvan2_1}
		|\nabla^2 f(\bx)| \lesssim |\bx|^p\,,
	\end{equation}
	for sufficiently large $|\bx|$, for some $p\in \mathbb{R}$. Then
	\begin{equation}
		|(f*\mu)(\bx)| \lesssim |\bx|^p\,,
	\end{equation}
	for sufficiently large $|\bx|$.
\end{lemma}
\begin{proof}
	This lemma can be proved similarly as the previous one, by using
	\begin{equation}
		(f*\mu)(\bx) =  \int_{\supp\mu} (f(\bx-\by)-f(\bx) + \nabla f(\bx)\cdot \by)\rd{\mu(\by)}\,,
	\end{equation}
	and applying Taylor's formula to $t\mapsto f(\bx-t\by)$ to get
	\begin{equation}
		f(\bx-\by) = f(\bx) - \nabla f(\bx)\cdot \by + \frac{1}{2}\by^\top \nabla^2 f(\bx-t_\by \by)\by\,.
	\end{equation}	
\end{proof}
}

\begin{proof}[Proof of Theorem \ref{thm_convFou2}]
	Let $\mu$ be a compactly supported signed measure with $E[|\mu|]<\infty$ and $\int_{\mathbb{R}^d}\rd{\mu(\bx)}=0$. We first assume $\mu$ is a smooth function. Denote 
	\begin{equation}
		g = -\Delta^{-1}\mu = -\frac{1}{|S^{d-1}|} \cdot \frac{|\bx|^{2-d}}{2-d}*\mu\,.
	\end{equation}
	Then $g$ is a locally integrable smooth function satisfying
	\begin{equation}
		|g(\bx)| \lesssim |\bx|^{1-d},\quad |\nabla g(\bx)| \lesssim |\bx|^{-d}\,,
	\end{equation}
	for large $|\bx|$, due to $\int_{\mathbb{R}^d}\rd{\mu(\bx)}=0$. {Indeed, these two estimates follow by applying Lemma \ref{lem_convvan} with $f$ being $|\bx|^{2-d}$ and $\nabla|\bx|^{2-d}=C\bx|\bx|^{-d}$, respectively.}
	
	We also have
	\begin{equation}
		|(W*\mu)(\bx)|\lesssim |\bx|^{b-1}\,,
	\end{equation}
	due to $|\nabla W(\bx)|\lesssim |\bx|^{b-1}$ in \eqref{thm_convFou2_1} and $\int_{\mathbb{R}^d}\rd{\mu(\bx)}=0$, {in view of Lemma \ref{lem_convvan}}.
	
	Then we claim that
	\begin{equation}
		\int_{\mathbb{R}^d} (W*\mu)\mu\rd{\bx} =  -\int_{\mathbb{R}^d} (\Delta W*\nabla g)\cdot\nabla g\rd{\bx}\,,
	\end{equation}
	where the LHS is well-defined since $E[|\mu|]<\infty$. To see this, we start by
	\begin{equation}\begin{split}
			(\Delta W*\partial_j g )(\bx)= & \int_{\mathbb{R}^d} \Delta W(\bx-\by)\partial_j g(\by)\rd{\by} = \int_{\mathbb{R}^d} \nabla W(\bx-\by)\cdot\nabla \partial_j g(\by)\rd{\by} \\
			= & \int_{\mathbb{R}^d} \partial_j\nabla W(\bx-\by)\cdot\nabla  g(\by)\rd{\by} =  \int_{\mathbb{R}^d} \partial_jW(\bx-\by)\Delta  g(\by)\rd{\by}  \\
			=&  -\int_{\mathbb{R}^d} \partial_j W(\bx-\by)\mu(\by)\rd{\by}\,, \\
	\end{split}\end{equation}
	where the integration by parts (three times) are justified by using the estimate $|\nabla W(\bx-\by) \otimes\nabla g(\by)| \lesssim |\by|^{b-1}\cdot |\by|^{-d},\,b-1-d<-d+1$ for fixed $\bx$ and large $|\by|$. {Indeed, the first integration by parts can be justified by
	\begin{equation}\begin{split}
		\int_{\mathbb{R}^d} & \Delta W(\bx-\by)\partial_j g(\by)\rd{\by} = 
		\lim_{R\rightarrow\infty} \int_{B(0;R)} \Delta W(\bx-\by)\partial_j g(\by)\rd{\by} \\
		= &  \lim_{R\rightarrow\infty} -\int_{\partial B(0;R)} \nabla W(\bx-\by) \partial_j g(\by) \cdot {\bf n}\rd{S(\by)}  + \int_{B(0;R)} \nabla  W(\bx-\by)\cdot \partial_j \nabla g(\by)\rd{\by}\,.
	\end{split}\end{equation}
	Here, $\nabla W(\bx-\by) \partial_j g(\by)=O(R^{b-1-d})$, and thus
	\begin{equation}
		\int_{\partial B(0;R)} \nabla W(\bx-\by) \partial_j g(\by) \cdot {\bf n}\rd{S(\by)} = O(R^{b-1-d + (d-1)})\,,
	\end{equation} 
	which vanishes in the limit $R\rightarrow\infty$. Therefore this integration by parts is justified. The other two integration by parts can be justified similarly.
}
	
	As a consequence, we obtain
	\begin{equation}
		\Delta W*\nabla g = -\nabla W * \mu\,.
	\end{equation}
	Finally we integrate by parts to get
	\begin{equation}
		\int_{\mathbb{R}^d} (\nabla W * \mu)\cdot\nabla g\rd{\bx} = -\int_{\mathbb{R}^d} ( W * \mu)\Delta g\rd{\bx}= \int_{\mathbb{R}^d} ( W * \mu)\mu\rd{\bx}\,,
	\end{equation}
	using the estimate $|(W*\mu)(\bx)\nabla g(\bx)| \lesssim |\bx|^{b-1}\cdot |\bx|^{-d}$. This proves the claim.
	
	Then, applying Theorem \ref{thm_convFou1} to the potential $-\Delta W$ and the compactly supported signed measure $\Phi_R\partial_j g$ where $\Phi_R$ is a truncation function, we get
	\begin{equation}\begin{split}
			-\int_{\mathbb{R}^d} (\Delta W*(\Phi_R\partial_j g))\cdot(\Phi_R\partial_j g)\rd{\bx} = & \int_{\mathbb{R}^d\backslash\{0\}} \cF[-\Delta W](\xi) |\cF[\Phi_R\partial_j g]|^2\rd{\xi} \\
			= & \int_{\mathbb{R}^d\backslash\{0\}} \cF[-\Delta W](\xi) |2\pi\xi_j\hat{g}*\hat{\Phi}_R|^2\rd{\xi}\,. \\
	\end{split}\end{equation}
	Using $|\nabla g(\bx)|\lesssim|\bx|^{-d}$ and $|\Delta W(\bx)|\lesssim |\bx|^{b-2}$ in \eqref{thm_convFou2_1}, one can justify that the LHS converges to $-\int_{\mathbb{R}^d} (\Delta W*\partial_j g)\cdot\partial_j g\rd{\bx}$ as $R\rightarrow\infty$. On the RHS, notice that {the Fourier transform of the relation $-\Delta g = \mu$ gives $4\pi^2|\xi|^2 \hat{g} = \hat{\mu}$, which leads to}
	\begin{equation}
		2\pi\xi_j\hat{g} = \frac{1}{2\pi}\cdot\frac{\xi_j}{|\xi|^2}\hat{\mu},\quad \xi\ne 0\,,
	\end{equation}
	and $\hat{\mu}$ is a smooth function with rapid decay at infinity, with $\hat{\mu}(0)=0$. Therefore $\hat{\mu}(\xi)\approx \nabla \hat{\mu}(0)\cdot \xi$ near 0, and thus $2\pi\xi_j\hat{g}$ is a bounded function, smooth on $\mathbb{R}^d\backslash \{0\}$, with rapid decay at infinity. {Then, since $-\Delta W$ satisfies the assumptions of Theorem \ref{thm_convFou1},  $\cF[-\Delta W]$ satisfies $\int_{B(0;1)^c}|\cF[-\Delta W](\xi)|\cdot |\xi|^{-k}\rd{\bx}<\infty$ for sufficiently large $k$}, as discussed in \eqref{hatWb}. Therefore, {we may proceed similarly as in the proof of Theorem \ref{thm_convFou1} to obtain}
	\begin{equation}\begin{split}
			\lim_{R\rightarrow\infty} & \int_{\mathbb{R}^d\backslash\{0\}} \cF[-\Delta W](\xi) |2\pi\xi_j\hat{g}*\hat{\Phi}_R|^2\rd{\xi} =  \int_{\mathbb{R}^d\backslash\{0\}} \cF[-\Delta W](\xi) |2\pi\xi_j\hat{g}|^2\rd{\xi} \\
			= & \int_{\mathbb{R}^d\backslash\{0\}}4\pi^2|\xi|^2\hat{W}(\xi) \Big|\frac{1}{2\pi}\cdot\frac{\xi_j}{|\xi|^2}\hat{\mu}(\xi)\Big|^2\rd{\xi} 
			=  \int_{\mathbb{R}^d\backslash\{0\}}\frac{\xi_j^2}{|\xi|^2}\hat{W}(\xi) |\hat{\mu}(\xi)|^2\rd{\xi}\,. \\
	\end{split}\end{equation}
	Then summing in $j$, the last quantity becomes $\int_{\mathbb{R}^d\backslash\{0\}}\hat{W}(\xi) |\hat{\mu}(\xi)|^2\rd{\xi}$, which finishes the proof of \eqref{EFou} in case $\mu$ is a smooth function.
	
	{
	For a general signed measure $\mu$, one can first mollify it into $\mu*\phi_\epsilon$, so that the latter satisfies \eqref{EFou}, i.e., 
	\begin{equation}
		\int_{\mathbb{R}^d} (W*(\mu*\phi_\epsilon))(\mu*\phi_\epsilon)\rd{\bx} = \int_{\mathbb{R}^d\backslash\{0\}}\hat{W}|\hat{\mu}|^2|\hat{\phi}_\epsilon|^2\rd{\xi}.
	\end{equation}
	Then sending $\epsilon\rightarrow 0$, similar to the proof of Theorem \ref{thm_convFou1}, we first use {\bf (W1)} to derive that $\lim_{\epsilon\rightarrow0}\int_{\mathbb{R}^d} (W*(\mu*\phi_\epsilon))(\mu*\phi_\epsilon)\rd{\bx} = \int_{\mathbb{R}^d} (W*\mu)\mu\rd{\bx}$. Then, since $\hat{\mu}(\xi)\approx \nabla \hat{\mu}(0)\cdot \xi$ near 0 and we assumed that  $-\Delta W$ satisfies the assumptions of Theorem \ref{thm_convFou1}, the function $\hat{W}|\hat{\mu}|^2 = \cF[-\Delta W]\frac{1}{4\pi^2|\xi|^2}|\hat{\mu}|^2$ is locally integrable on $\mathbb{R}^d$. If $\cF[-\Delta W]\in L^1(\mathbb{R}^d)$ or $\cF[-\Delta W]$ is one-signed on $B(0;R_0)^c$, then the same is true for $\hat{W}|\hat{\mu}|^2$. Therefore we derive $\lim_{\epsilon\rightarrow 0}\int_{\mathbb{R}^d}\hat{W}|\hat{\mu}|^2|\hat{\phi}_\epsilon|^2\rd{\xi} = \int_{\mathbb{R}^d}\hat{W}|\hat{\mu}|^2\rd{\xi}$ in the same way as in the proof of Theorem \ref{thm_convFou1}, and thus obtain \eqref{EFou} for $\mu$.
}
	
\end{proof}

\begin{theorem}\label{thm_convFou4}
	Assume $W$ satisfies {\bf (W0-w)}{\bf (W1)}, and $\Delta^2 W$ satisfies the assumptions of Theorem \ref{thm_convFou1}. Assume 
	\begin{equation}\label{thm_convFou4_1}
		|\nabla W(\bx)|\lesssim |\bx|^{b-1},\quad |\nabla^2 W(\bx)|\lesssim |\bx|^{b-2},\quad |\nabla^3 W(\bx)|\lesssim |\bx|^{b-3},\quad |\Delta^2 W(\bx)|\lesssim |\bx|^{b-4}\,,
	\end{equation}
	for some $b < 4$ and large $|\bx|$. Then $W$ is Fourier representable at level 2.
\end{theorem}

\begin{proof}
	Let $\mu$ be a compactly supported signed measure with $E[|\mu|]<\infty$ and $\int_{\mathbb{R}^d}\rd{\mu(\bx)}=\int_{\mathbb{R}^d}\bx\rd{\mu(\bx)}=0$. We first assume $\mu$ is a smooth function. Define $g$ as in the proof of Theorem \ref{thm_convFou2}, which satisfies
	\begin{equation}
		|g(\bx)| \lesssim |\bx|^{-d},\quad |\nabla g(\bx)| \lesssim |\bx|^{-d-1}\,,
	\end{equation}
	for large $|\bx|$, due to $\int_{\mathbb{R}^d}\rd{\mu(\bx)}=\int_{\mathbb{R}^d}\bx\rd{\mu(\bx)}=0$, {in view of Lemma \ref{lem_convvan2}}.  We also have
	\begin{equation}
		|(W*\mu)(\bx)|\lesssim |\bx|^{b-2},\quad |(\nabla W*\mu)(\bx)|\lesssim |\bx|^{b-3}\,,
	\end{equation}
	by $|\nabla^2 W(\bx)|\lesssim |\bx|^{b-2}$ and $|\nabla^3 W(\bx)|\lesssim |\bx|^{b-3}$, {in view of Lemma \ref{lem_convvan2}}.
	
	Then we claim that
	\begin{equation}
		\int_{\mathbb{R}^d} (W*\mu)\mu\rd{\bx} =  \int_{\mathbb{R}^d} (\Delta^2 W* g) g\rd{\bx}\,.
	\end{equation}
	To see this, we start by
	\begin{equation}\begin{split}
		(\Delta^2 W* g)(\bx) = & \int_{\mathbb{R}^d} \Delta^2 W(\bx-\by)g(\by)\rd{\by} = \int_{\mathbb{R}^d} \nabla\Delta W(\bx-\by)\cdot \nabla g(\by)\rd{\by}\\ = & \int_{\mathbb{R}^d} \Delta W(\bx-\by)\Delta g(\by)\rd{\by}\,,
	\end{split}\end{equation}
	where the integration by parts (two times) are justified by using the estimates $|\nabla\Delta W(\bx-\by) g(\by)| \lesssim |\by|^{b-3}\cdot |\by|^{-d},\,b-3-d<-d+1$ and $|\Delta W(\bx-\by) \nabla g(\by)| \lesssim |\by|^{b-2}\cdot |\by|^{-d-1},\,b-2-d-1<-d+1$ for fixed $\bx$ and large $|\by|$. Then we integrate by parts to get
	\begin{equation}\begin{split}
			\int_{\mathbb{R}^d} (\Delta^2 W* g) g\rd{\bx} = &  \int_{\mathbb{R}^d} (\Delta W* \Delta g) g\rd{\bx}=  \int_{\mathbb{R}^d} (\Delta W* \mu) g\rd{\bx} \\
			= &  -\int_{\mathbb{R}^d} (\nabla W*\mu) \cdot \nabla g \rd{\bx} = \int_{\mathbb{R}^d} (W*\mu) \Delta g\rd{\bx}= \int_{\mathbb{R}^d} (W*\mu) \mu\rd{\bx}\,,
	\end{split}\end{equation}
	where the integration by parts (two times) are justified by using the estimates $|(\nabla W*\mu)g| \lesssim |\bx|^{b-3}\cdot |\bx|^{-d}$  and $|(W*\mu)\nabla g| \lesssim |\bx|^{b-2}\cdot |\bx|^{-d-1}$. This proves the claim.
	
	Then, similar to the proof of Theorem \ref{thm_convFou2}, we may calculate
	\begin{equation}\begin{split}
			\int_{\mathbb{R}^d} (\Delta^2 W*(\Phi_R g))\cdot(\Phi_R g)\rd{\bx} 
			= & \int_{\mathbb{R}^d\backslash\{0\}} \cF[\Delta^2 W](\xi) |\hat{g}*\hat{\Phi}_R|^2\rd{\xi}\,. \\
	\end{split}\end{equation}
	Notice that 
	\begin{equation}
		\hat{g} = \frac{1}{4\pi^2|\xi|^2}\hat{\mu},\quad \xi\ne 0\,,
	\end{equation}
	and $\hat{\mu}$ is a smooth function with rapid decay at infinity, with $\hat{\mu}=\nabla\hat{\mu}=0$. Therefore $\hat{\mu}(\xi)\approx \frac{1}{2}\xi^\top\nabla^2 \hat{\mu}(0)\cdot \xi$ near 0, and thus $\hat{g}$ is a bounded function, smooth on $\mathbb{R}^d\backslash \{0\}$, with rapid decay at infinity. Then one can pass to the limit $R\rightarrow\infty$ similar to the proof of Theorem \ref{thm_convFou2} to get
	\begin{equation}
		\int_{\mathbb{R}^d} (\Delta^2 W* g) g\rd{\bx} = \int_{\mathbb{R}^d\backslash\{0\}} \cF[\Delta^2 W](\xi) |\hat{g}(\xi)|^2\rd{\xi} = \int_{\mathbb{R}^d\backslash\{0\}} \hat{W}(\xi) |\hat{\mu}(\xi)|^2\rd{\xi}\,,
	\end{equation}
	which finishes the proof for smooth $\mu$. The case of a general $\mu$ can be treated by mollification trick similar to the proof of Theorem \ref{thm_convFou2}.
	
\end{proof}

\subsection{A dimension reduction trick}

The results in the previous two sections apply to the repulsive power-law potentials $-\frac{|\bx|^b}{b}$ with $-d<b<4$ for $d\ge 3$, but may not apply for some cases in dimensions 1 and 2. This motivates the following theorem.

\begin{theorem}\label{thm_convFou3}
	Assume $W$ is a tempered distribution satisfying {\bf (W0-w)}, and $\hat{W}$ is a locally integrable function on $\mathbb{R}^d\backslash \{0\}$. If there exists a function $\tilde{W}$ on $\mathbb{R}^{D},\,D>d$ satisfying {\bf (W0-w)} with
	\begin{equation}\label{thm_convFou3_1}
		\int_{\mathbb{R}^{D-d}}\hat{\tilde{W}}(\xi_1,\dots,\xi_D)\rd{\xi_{d+1}}\cdots\rd{\xi_D} = \hat{W}(\xi_1,\dots,\xi_d),\quad \forall 0\ne(\xi_1,\dots,\xi_d)\in\mathbb{R}^d\,,
	\end{equation}
	and
	\begin{equation}\label{thm_convFou3_2}
		\tilde{W}(x_1,\dots,x_d,0,\dots,0)=W(x_1,\dots,x_d)\,,
	\end{equation}
	then the Fourier representable property for $\tilde{W}$ at level $k$ implies the same property for $W$.
\end{theorem}

Notice that \eqref{thm_convFou3_1} and \eqref{thm_convFou3_2} are formally equivalent.

\begin{proof}
	Let $\mu\ne 0$ be a compactly supported signed measure on $\mathbb{R}^d$ with $E[|\mu|]<\infty$ and $\int_{\mathbb{R}^d}\bx^{\otimes j}\rd{\mu(\bx)} = 0,\,j=0,1,\dots,k-1$. Then, denote
	\begin{equation}
		\tilde{\mu} = \mu(x_1,\dots,x_d) \otimes \delta(x_{d+1},\dots,x_D)\,,
	\end{equation} 
	which is a compactly supported signed measure on $\mathbb{R}^D$ with $\int_{\mathbb{R}^D}\bx^{\otimes j}\rd{\tilde{\mu}(\bx)} = 0,\,j=0,1,\dots,k-1$. We have $\tilde{E}[|\tilde{\mu}|] = E[|\mu|]<\infty$ due to \eqref{thm_convFou3_2}. Then the Fourier representable property at level $k$ for $\tilde{W}$ gives
	\begin{equation}
		\int_{\mathbb{R}^D} (\tilde{W}*\tilde{\mu})\tilde{\mu}\rd{\bx} = \int_{\mathbb{R}^D\backslash \{0\}} \hat{\tilde{W}}(\xi)|\hat{\tilde{\mu}}(\xi)|^2\rd{\xi}\,.
	\end{equation}
	Notice that the LHS is exactly $\int_{\mathbb{R}^d} (W*\mu)\mu\rd{\bx}$ due to \eqref{thm_convFou3_2}. For the RHS, we have
	\begin{equation}
		\hat{\tilde{\mu}}(\xi_1,\dots,\xi_D) = \hat{\mu}(\xi_1,\dots,\xi_d)\,.
	\end{equation}
	Therefore, using \eqref{thm_convFou3_1}, we see that it is $\int_{\mathbb{R}^d\backslash\{0\}} \hat{W}(\xi)|\hat{\mu}(\xi)|^2\rd{\xi}$. This gives the conclusion.
	
\end{proof}

\subsection{Application to power-law potentials}

We first calculate the Fourier transform of power-law potentials as tempered distributions.
\begin{theorem}\label{thm_powerF}
	If $0<\Re(s)<d$, then 
	\begin{equation}\label{thm_powerF_1}
		\cF\Big[\frac{|\bx|^{-s}}{s}\Big](\xi) =  \cpf(s)|\xi|^{-d+s},\quad \cpf(s) :=\pi^{s-\frac{d}{2}}\frac{\Gamma(\frac{d-s}{2})}{\Gamma(\frac{s}{2})s}\,,
	\end{equation}
	as a locally integrable function. If $\Re(s)\le 0$ (with $\frac{|\bx|^{-s}}{s}$ viewed as $-\ln|\bx|$ when $s=0$), then $\cF[\frac{|\bx|^{-s}}{s}]$ is a locally integrable function on $\mathbb{R}^d\backslash \{0\}$, and the above formula is also true for any $\xi\ne 0$.
	
	There also holds
	\begin{equation}\label{thm_powerF_2}
	\cF\Big[-\frac{|\bx|^{-s}}{s}\ln|\bx|\Big](\xi) =  \tcpf(s)|\xi|^{-d+s} + \cpf(s)|\xi|^{-d+s}\ln |\xi|,\quad \tcpf(s) := \frac{\cpf(s)}{s}+\cpf'(s)\,,
	\end{equation}
	for any $\Re(s)<d,\,s\ne 0$ in the same sense.
\end{theorem}

\begin{remark}[The function $\cpf(s)$]\label{rmk_cpf}
	On $\{s\in\mathbb{C}:\Re(s) < d\}$, the function $\cpf(s)$ is holomorphic, with simple zeros at $s=-2,-4,-6,\dots$. On the real line, $\cpf>0$ on the intervals $(-2,d),\,(-6,-4),\,(-10,-8),\dots$ and $\cpf<0$ on $(-4,-2),\,(-8,-6),\dots$. $\cpf$ has a simple pole at $s=d$ with $\cpf(s) \approx \pi^{\frac{d}{2}}\frac{2}{\Gamma(\frac{d}{2})d}\cdot \frac{1}{d-s}$ for $s$ near $d$.
	
	We also notice that for $s$ close to $-2$,
	\begin{equation}
		\frac{1}{\Gamma(\frac{s}{2})} =  \frac{\frac{s}{2}\cdot \frac{s+2}{2} }{\Gamma(\frac{s+4}{2})} \approx -\frac{1}{2}(s+2)\,,
	\end{equation}
	to the leading order, and therefore
	\begin{equation}
		\cpf(s)  \approx \pi^{-2-\frac{d}{2}}\frac{\Gamma(\frac{d+2}{2})}{4}(s+2)\,,
	\end{equation}
	i.e.,
	\begin{equation}
		\cpf'(-2)  \approx \pi^{-2-\frac{d}{2}}\frac{\Gamma(\frac{d+2}{2})}{4} > 0\,.
	\end{equation}
\end{remark}

\begin{proof}
	\eqref{thm_powerF_1} in the case $0<\Re(s)<d$ is well-known. For $\Re(s)\le 0$, it suffices to show that
	\begin{equation}\label{xsphi}
		\int_{\mathbb{R}^d}\frac{|\bx|^{-s}-1}{s}\phi(\bx)\rd{\bx} = \pi^{s-\frac{d}{2}}\frac{\Gamma(\frac{d-s}{2})}{\Gamma(\frac{s}{2})s} \int_{\mathbb{R}^d\backslash\{0\}}|\xi|^{-d+s}\bar{\hat{\phi}}(\xi)\rd{\xi}\,,
	\end{equation}
	for any complex-valued Schwartz function $\phi$ such that $0\notin \supp\hat{\phi}$. Notice that the extra $-1$ on the LHS gives no contribution because $0\notin \supp\hat{\phi}$ implies that $\phi$ is mean-zero. Also notice that the case $s=0$ can be derived using a limit procedure $s\rightarrow 0$ in the above expression.
	
	Fix $\phi$. The LHS of \eqref{xsphi} is a holomorphic function in $\Re(s)<d$ (which is $\int_{\mathbb{R}^d}(-\ln|\bx|)\phi(\bx)\rd{\bx}$ when $s=0$). Since the gamma function has no zero and has simple poles at $0,-1,-2,\dots$, we see that the RHS is also a holomorphic function in $\Re(s)<d$, which is understood as $0$ when $s=-2,-4,-6,\dots$ and $\pi^{-\frac{d}{2}}\frac{\Gamma(\frac{d}{2})}{2} \int_{\mathbb{R}^d\backslash\{0\}}|\xi|^{-d}\bar{\hat{\phi}}(\xi)\rd{\xi}$ when $s=0$. Since we already know that \eqref{xsphi} is true for any $0<\Re(s)<d$, it is also true for any $\Re(s)<d$ by analytic continuation.
	
	\eqref{thm_powerF_2} can be obtained from \eqref{thm_powerF_1} by differentiating with respect to $s$. {In fact, formally differentiating \eqref{thm_powerF_1} leads to
	\begin{equation}
		\cF\Big[-\frac{|\bx|^{-s}\ln|\bx|}{s}-\frac{|\bx|^{-s}}{s^2}\Big](\xi) =  \cpf'(s)|\xi|^{-d+s} + \cpf(s)|\xi|^{-d+s}\ln|\xi|\,.
	\end{equation}
	Then doing a linear combination with \eqref{thm_powerF_1} gives \eqref{thm_powerF_2}. Here, when $0<\Re(s)<d$, the differentiation with respect to $s$ inside $\cF[\cdot]$ can be justified by taking a Schwartz test function $\phi$. When $\Re(s)\le 0,\,s\ne 0$, we further impose  $0\notin \supp\hat{\phi}$, and thus justify \eqref{thm_powerF_2} on $\mathbb{R}^d\backslash\{0\}$.
	}
	
\end{proof}

Then we give the Fourier representable property for power-law and logarithmic power-law potentials.
\begin{theorem}\label{thm_Fourep}
	The repulsive power-law potential $W_b(\bx)=-\frac{|\bx|^b}{b}$ is
	\begin{itemize}
		\item Fourier representable at level 0, if $-d<b<0$;
		\item Fourier representable at level 1, if $0\le b < 2$;
		\item Fourier representable at level 2, if $2\le b < 4$.
	\end{itemize}
	
	The attractive-repulsive power-law potential $W_{a,b}(\bx)=\frac{|\bx|^a}{a}-\frac{|\bx|^b}{b}$ is Fourier representable at level 2 if $-d<b<a<4$.
	
	The logarithmic power-law potential $W_{b,\ln} = \frac{|\bx|^b}{b}\ln|\bx|$ is 
	\begin{itemize}
	\item Fourier representable at level 0, if $-d<b<0$;
	\item Fourier representable at level 1, if $0< b < 2$;
	\item Fourier representable at level 2, if $2\le b < 4$.
\end{itemize}
\end{theorem}

\begin{proof}
	We first consider $W_b(\bx)=-\frac{|\bx|^b}{b}$. If $-d<b<0$, then Theorem \ref{thm_convFou1} applies, showing that $W_b$ is Fourier representable at level 0. Now we first assume $d\ge 3$. If $0\le b < 2$, then Theorem \ref{thm_convFou2} applies since
	\begin{equation}
		-\Delta W_b = (b-2+d)|\bx|^{b-2}\,,
	\end{equation}
	showing that $W_b$ is Fourier representable at level 1. If $2\le b < 4$, then Theorem \ref{thm_convFou4} applies since
	\begin{equation}
		\Delta^2 (-W_b) = (b-2+d)(b-2)(b-4+d)|\bx|^{b-4}\,,
	\end{equation}	
	showing that $-W_b$ is Fourier representable at level 2. The same is true for $W_b$ because $E_{-W_b}[|\mu|]<\infty$ for any compactly supported signed measure $\mu$. For the case $0\le b < 4$ with $d=1,2$, the same results can be obtained by using Theorem \ref{thm_convFou3} with $D=3$ and $\tilde{W}(\bx)=-\frac{|\bx|^b}{b}$, for which the conditions \eqref{thm_convFou3_1} and \eqref{thm_convFou3_2} can be verified explicitly via Theorem \ref{thm_powerF}.
	
	For $W_{a,b}(\bx)=\frac{|\bx|^a}{a}-\frac{|\bx|^b}{b}$ with $-d<b<a<4$, we notice that its singularity near 0 is comparable to $W_b$. Therefore $E_{a,b}[|\mu|]<\infty$ implies $E_b[|\mu|]<\infty$ (where $E_b$ denotes the energy associated to $W_b$), which further implies $E_a[|\mu|]<\infty$. Therefore, applying the Fourier representable property at level 2 for $W_b$ and $W_a$ and noticing that $W_{a,b}= W_b-W_a$, we see that $W_{a,b}$ is Fourier representable at level 2.
	
	For $W_{b,\ln} = \frac{|\bx|^b}{b}\ln|\bx|$, one can prove its Fourier representable property similarly as $W_b$. In fact, the case $-d<b<0$ follows from Theorem \ref{thm_convFou1}. Now we first assume $d\ge 3$. The case $0<b<2$ follows from Theorem \ref{thm_convFou1} and
	\begin{equation}
		-\Delta W_{b,\ln} = -(b+d-2)|\bx|^{b-2}\ln|\bx| - \frac{1}{b}(2b+d-2)|\bx|^{b-2}\,.
	\end{equation}
	The case $2\le b<4$ follows from Theorem \ref{thm_convFou2} and
	\begin{equation}\begin{split}
		\Delta^2 W_{b,\ln} = & (b-2+d)(b-2)(b-4+d)|\bx|^{b-4}\ln|\bx| \\
		& + \Big((b+d-2)(2b+d-6)+\frac{1}{b}(2b+d-2)(b-2)(b+d-4)\Big) |\bx|^{b-4}\,.
	\end{split}\end{equation}
	For the case $d=1,2$ with $0<b<4$, the same results can be obtained by using Theorem \ref{thm_convFou3}.
\end{proof}

\begin{remark}
	We remark that the repulsive or attractive power-law potential $W(\bx)=\pm \frac{|\bx|^b}{-b}$ with $b\ge 4$ is not Fourier representable at level 2. In fact, if $b\ge 4$ is not an even integer, then by Theorem \ref{thm_powerF} we see that $\hat{W}(\xi)= \pm\cpf(-b)|\xi|^{-d-b}$ where $-d-b<-d-4$ and $\cpf(-b)\ne 0$. For a compactly supported signed measure $\mu$ with $\int_{\mathbb{R}^d}\rd{\mu(\bx)}=\int_{\mathbb{R}^d}\bx\rd{\mu(\bx)}=0$, one generally has $\hat{\mu}(\xi) = O(|\xi|^2)$ for small $|\xi|$. Therefore the integrand in \eqref{EFou} is generally not locally integrable at 0, making the integral undefined. If $b=4,6,8,\dots$, then $\hat{W}=0$ on $\mathbb{R}^d\backslash \{0\}$, and thus the RHS of \eqref{EFou} is 0. However, it is easy to see that $E[\mu]$ is not always zero, for example, by considering $\mu = \delta_{-1}-2\delta_0+\delta_1$. Similarly $W_{b,\ln} = \frac{|\bx|^b}{b}\ln|\bx|$ with $b\ge 4$  is not Fourier representable at level 2.
\end{remark}

\subsection{Fourier representation of the LIC radius}

\begin{theorem}\label{thm_RLICab}
	For the potentials $W_{a,b}$ with $-d<b<a$ and $W_{b,\ln}$ with $b>-d,\,b\ne 0$, its LIC radius is given by 
	\begin{equation}\label{thm_RLICab_1}\begin{split}
		R_{\textnormal{LIC}}[W] & = \sup\Big\{R\ge 0:  \textnormal{ For any nontrivial signed measure $\mu$ supported on $\overline{B(0;R)}$} \\ & \textnormal{ with $E[|\mu|]<\infty$ and $\int_{\mathbb{R}^d}\rd{\mu(\bx)}=\int_{\mathbb{R}^d}\bx\rd{\mu(\bx)}=0$, there holds }  
		E[\mu]>0\Big\}\,.
	\end{split}\end{equation}
	If one further assumes $a<4$ for $W_{a,b}$ or $b<4$ for $W_{b,\ln}$, then
	\begin{equation}\label{thm_RLICab_2}\begin{split}
		R_{\textnormal{LIC}}[W] & = \sup\Big\{R\ge 0:  \textnormal{ For any nontrivial signed measure $\mu$ supported on $\overline{B(0;R)}$}  \\ & \textnormal{ with $E[|\mu|]<\infty$ and $\int_{\mathbb{R}^d}\rd{\mu(\bx)}=\int_{\mathbb{R}^d}\bx\rd{\mu(\bx)}=0$,} \\
		&  \textnormal{ there holds } \int_{\mathbb{R}^d\backslash \{0\}} \hat{W}(\xi)|\hat{\mu}(\xi)|^2\rd{\xi}>0\Big\}\,.
	\end{split}\end{equation}
\end{theorem}

\begin{proof}
	It is clear that these potentials satisfy {\bf (W0)}, {\bf (W1)}. We claim that they satisfy the property:
	\begin{equation}\label{claim_rho01}\begin{split}
	& \text{For any compactly supported $\rho_0,\rho_1\in\cM(\mathbb{R}^d)$ with $E[\rho_i]<\infty,\,i=0,1$,}\\
	& \text{there holds $E[\rho_0+\rho_1] < \infty$.	}
	\end{split}\end{equation}
	
	To prove the claim for $W_{a,b}$ with $-d<b<a$, it suffices to treat $W_b(\bx)=-\frac{|\bx|^b}{b}$ since its singularity at 0 is comparable to $W_{a,b}$ . The case $b>0$ is trivial because $W_b$ is continuous. For $-d<b\le 0$, $W_b$ is Fourier representable at level 1 by Theorem \ref{thm_Fourep}. Take a mollifier $\phi_\epsilon$. Then $(\rho_1-\rho_0)*\phi_\epsilon$ is a compactly supported mean-zero smooth function. Therefore its Fourier representation for $W_b$ gives
	\begin{equation}
		E_b [(\rho_1-\rho_0)*\phi_\epsilon] = \frac{1}{2}\int_{\mathbb{R}^d\backslash\{0\}} \hat{W}_b(\xi)|\hat{\rho}_1(\xi)-\hat{\rho}_0(\xi)|^2|\hat{\phi}_\epsilon(\xi)|^2\rd{\xi} \ge 0\,,
	\end{equation}
	since $\hat{W}_b>0$. It follows that 
	\begin{equation}
		E[(\rho_0+\rho_1)*\phi_\epsilon] \le 2(E[\rho_0*\phi_\epsilon]+E[\rho_1*\phi_\epsilon])\,,
	\end{equation}
	by convexity. Then sending $\epsilon\rightarrow 0$, the RHS converges to $2(E[\rho_0]+E[\rho_1])$ by Lemma \ref{lem_W1}, which implies $E[\rho_0+\rho_1] < \infty$ due to Lemma \ref{lem_W2}.
	
	Then we prove the claim for $W_{b,\ln}$ with $b>-d,\,b\ne 0$. The case $b>0$ is trivial because $W_{b,\ln}$ is continuous. For $-d<b<0$, $W_{b,\ln}$ is Fourier representable at level 0 by Theorem \ref{thm_Fourep}. Therefore
	\begin{equation}
		E[\rho_i] = \frac{1}{2}\int_{\mathbb{R}^d\backslash\{0\}} \hat{W}(\xi)|\hat{\rho}_i(\xi)|^2\rd{\xi},\quad E[\rho_i*\phi_\epsilon] = \frac{1}{2}\int_{\mathbb{R}^d\backslash\{0\}} \hat{W}(\xi)|\hat{\rho}_i(\xi)|^2|\hat{\phi}_\epsilon(\xi)|^2\rd{\xi}\,,
	\end{equation}
	and
	\begin{equation}\begin{split}
		E[(\rho_0+\rho_1)*\phi_\epsilon] = & \frac{1}{2}\int_{\mathbb{R}^d\backslash\{0\}} \hat{W}(\xi)|\hat{\rho}_0(\xi)+\hat{\rho}_1(\xi)|^2|\hat{\phi}_\epsilon(\xi)|^2\rd{\xi} \\
		\le & \frac{1}{2}\int_{\hat{W}>0} \hat{W}(\xi)|\hat{\rho}_0(\xi)+\hat{\rho}_1(\xi)|^2|\hat{\phi}_\epsilon(\xi)|^2\rd{\xi} \\
		\le & \int_{\hat{W}>0} \hat{W}(\xi)|\hat{\rho}_0(\xi)|^2|\hat{\phi}_\epsilon(\xi)|^2\rd{\xi} + \int_{\hat{W}>0} \hat{W}(\xi)|\hat{\rho}_1(\xi)|^2|\hat{\phi}_\epsilon(\xi)|^2\rd{\xi}\,.\\
	\end{split}\end{equation}
	Sending $\epsilon\rightarrow 0$, the RHS converges to $\int_{\hat{W}>0} \hat{W}(\xi)|\hat{\rho}_0(\xi)|^2\rd{\xi} + \int_{\hat{W}>0} \hat{W}(\xi)|\hat{\rho}_1(\xi)|^2\rd{\xi}$ which is finite since $E[\rho_i]$ are finite. Therefore $E[\rho_0+\rho_1] < \infty$ due to Lemma \ref{lem_W2}.
	
	Then we take any distinct $\rho_0,\rho_1\in\cM(\overline{B(0;R)})$ with $E[\rho_0]<\infty,\,E[\rho_1]<\infty$ and $\int_{\mathbb{R}^d}\bx\rd{\rho_0(\bx)}=\int_{\mathbb{R}^d}\bx\rd{\rho_1(\bx)}=0$. For the previously stated potentials, we have $E[\rho_0+\rho_1]<\infty$, which implies $E[|\rho_1-\rho_0|]<\infty$. Therefore one can calculate
	\begin{equation}
		\frac{\rd^2}{\rd t^2} E[(1-t)\rho_0+t\rho_1] = 2E[\rho_1-\rho_0]\,.
	\end{equation}
	Since $\mu=\rho_1-\rho_0$ is a nontrivial signed measure supported on $\overline{B(0;R)}$ with $E[|\mu|]<\infty$ and $\int_{\mathbb{R}^d}\rd{\mu(\bx)}=\int_{\mathbb{R}^d}\bx\rd{\mu(\bx)}=0$, we see that $R_{\textnormal{LIC}}[W]$ is at least the RHS of \eqref{thm_RLICab_1}. To see that $R_{\textnormal{LIC}}[W]$ is at most this quantity, we take any $\mu$ as described above, then one can write $\mu = C(\rho_1-\rho_0)$ where $C>0$ and $\rho_0,\rho_1\in\cM(\overline{B(0;R)})$ with $E[\rho_0]<\infty,\,E[\rho_1]<\infty$ and $\int_{\mathbb{R}^d}\bx\rd{\rho_0(\bx)}=\int_{\mathbb{R}^d}\bx\rd{\rho_1(\bx)}=0$. In fact, one can take $\rho_1=\frac{1}{C}(\mu_+ + f)$ and $\rho_0=\frac{1}{C}(\mu_- + f)$ where $f$ is a nonnegative continuous function supported on $\overline{B(0;R)}$ such that $\rho_0,\rho_1$ satisfy the condition on the center of mass, and $C>0$ is a normalizing constant. Then, using the definition of $R_{\textnormal{LIC}}[W]$, we see that $R_{\textnormal{LIC}}[W]$ is at most the RHS of \eqref{thm_RLICab_1}.
	
	Finally \eqref{thm_RLICab_2} follows from \eqref{thm_RLICab_1} and the Fourier representable property at level 2 for $W_{a,b}$ with $-d<b<a<4$ or $W_{b,\ln}$ with $-d<b<4,\,b\ne 0$, which was given by Theorem \ref{thm_Fourep}.
\end{proof}

\begin{remark}
	From the proof, it is clear that \eqref{thm_RLICab_1} holds for any potential that satisfies {\bf (W0)} and \eqref{claim_rho01}. The condition \eqref{claim_rho01} is indeed very mild. For example, the potentials discussed in Remark \ref{rmk_W1} satisfy this condition. Also, if one further assumes the Fourier representable property at level 2, then \eqref{thm_RLICab_2} holds.
\end{remark}

\subsection{Proof of Theorem \ref{thm_Wbln}}\label{sec_proof2ln}

	By Theorem \ref{thm_RLICab}, $R_\textnormal{LIC}[W_{2,\ln}]$ can be calculated by \eqref{thm_RLICab_2}. By Theorem \ref{thm_powerF} we have
	\begin{equation}
		\cF[W_{2,\ln}] = \tcpf(-2)|\xi|^{-d-2},\quad \xi\ne 0\,,
	\end{equation}
	where 
	\begin{equation}\label{eqtcpf2}
		\tcpf(-2) = \cpf'(-2) = \pi^{-2-\frac{d}{2}}\frac{\Gamma(\frac{d+2}{2})}{4} > 0\,,
	\end{equation}
	(c.f. Remark \ref{rmk_cpf}). It follows that $R_\textnormal{LIC}[W_{2,\ln}]=\infty$, i.e., $W_{2,\ln}$ is LIC, and thus its energy minimizer is unique up to translation. By Theorem \ref{thm_EL}, to find the energy minimizer for $W_{2,\ln}$, it suffices to find $\rho\in\cM(\mathbb{R}^d)$ such that the Euler-Lagrange condition \eqref{thm_EL_1}\eqref{thm_EL_2} holds. 
	
	We first assume $d=1$. \cite[Theorem 1]{Fra} showed that the minimizer for $W_{a,2}$ with $2<a<3$ is
	\begin{equation}
		C(R_a^2 - x^2)_+^{-\frac{a-1}{2}},\quad R_a = \Big( \frac{\Gamma(\frac{3-a}{2})\sin\big((a-1)\frac{\pi}{2}\big)}{\Gamma(\frac{4-a}{2})(a-1)\sqrt{\pi}}\Big)^{\frac{1}{a-2}}\,.
	\end{equation}
	This implies the Euler-Lagrange condition (noticing here that $W_{a,2}*(R_a^2 - x^2)_+^{-\frac{a-1}{2}}$ is continuous)
	\begin{equation}
		W_{a,2}*(R_a^2 - x^2)_+^{-\frac{a-1}{2}} = C_a,\quad \text{ on }[-R_a,R_a]\,,
	\end{equation}
	and 
	\begin{equation}
		W_{a,2}*(R_a^2 - x^2)_+^{-\frac{a-1}{2}} \ge C_a,\quad \text{ on }\mathbb{R}\,,
	\end{equation}
	for some $C_a\in\mathbb{R}$. Rescaling by $R_a$ and denoting
	\begin{equation}
		\tilde{W}_{a,2}(x) = \frac{|x|^a}{2} - \frac{a}{2R_a^{a-2}}\cdot \frac{|x|^2}{2}\,,
	\end{equation}
	we see that
	\begin{equation}\label{tWa21}
		\tilde{W}_{a,2}*(1 - x^2)_+^{-\frac{a-1}{2}} = \tilde{C}_a,\quad \text{ on }[-1,1]\,,	
	\end{equation}
	and 
	\begin{equation}\label{tWa22}
		\tilde{W}_{a,2}*(1 - x^2)_+^{-\frac{a-1}{2}} \ge \tilde{C}_a,\quad \text{ on 	}\mathbb{R}\,.
	\end{equation}
	By rewriting 
	\begin{equation}
		\tilde{W}_{a,2}(x) = \frac{|x|^a-|x|^2}{2} + \Big(1-\frac{a}{2R_a^{a-2}}\Big)\cdot \frac{|x|^2}{2}\,,
	\end{equation}
	we see that 
	\begin{equation}
		\lim_{a\rightarrow 2^+}\frac{1}{a-2}\tilde{W}_{a,2}(x) = \frac{|x|^2}{2}\ln|x| +  \lambda\frac{|x|^2}{2}\,,
	\end{equation}
	uniformly on compact sets for $x$, where
	\begin{equation}
		\lambda := \frac{\rd}{\rd{a}}\Big|_{a=2} \Big(1-\frac{a}{2R_a^{a-2}}\Big) = -\frac{1}{2} + \frac{\rd}{\rd{a}}\Big|_{a=2} \Big( \frac{\Gamma(\frac{3-a}{2})\sin\big((a-1)\frac{\pi}{2}\big)}{\Gamma(\frac{4-a}{2})(a-1)\sqrt{\pi}}\Big)\,.
	\end{equation}
	Then we may take the limit $a\rightarrow 2^+$ in \eqref{tWa21} and \eqref{tWa22} to get
	\begin{equation}
		\Big(\frac{|x|^2}{2}\ln|x| +  \lambda\frac{|x|^2}{2}\Big)*(1 - x^2)_+^{-\frac{1}{2}} = \tilde{C},\quad \text{ on }[-1,1]\,,
	\end{equation}
	and 
	\begin{equation}
		\Big(\frac{|x|^2}{2}\ln|x| +  \lambda\frac{|x|^2}{2}\Big)*(1 - x^2)_+^{-\frac{1}{2}} \ge \tilde{C},\quad \text{ on 		}\mathbb{R}\,.
	\end{equation}
	Rescaling by $e^{-\lambda}$, we get \eqref{thm_EL_1}\eqref{thm_EL_2} for the potential $\frac{|x|^2}{2}\ln|x|$ and the probability measure $C(R^2-x^2)_+^{-1/2}$ with $R=e^\lambda$, which gives the conclusion for $d=1$.
	
	For $d\ge 2$, \cite[Theorem 1]{FM} showed that the minimizer for $W_{a,2}$ with $2<a<4$ is $\delta_{\partial B(0;R_a)}$ where
	\begin{equation}
		R_a = \Big(\frac{\Gamma(\frac{d+1}{2})\Gamma(\frac{2d+a-2}{2})}{2^{a-2}\Gamma(\frac{d+a-1}{2})\Gamma(d)}\Big)^{\frac{1}{a-2}}\,.
	\end{equation}
	The same rescaling and limiting procedure as before would yield the conclusion.
	
\section{Poincar\'e-type inequalities for signed measures}\label{sec_poin}

Using Theorem \ref{thm_RLICab}, we reduced the study of the LIC radius for $W_{a,b}$ with $-d<b<a<4$ to determining the sign of $\int_{\mathbb{R}^d\backslash \{0\}} \hat{W}_{a,b}(\xi)|\hat{\mu}(\xi)|^2\rd{\xi}$ for compactly supported signed measures $\mu$ with vanishing moments. $\hat{W}_{a,b}$ is again the difference of two power functions in $|\xi|$. This motivates us to compare the sizes of $\int_{\mathbb{R}^d\backslash \{0\}} |\xi|^{-s}|\hat{\mu}(\xi)|^2\rd{\xi}$ for different values of $s$, i.e., compare the negative Sobolev norms of $\mu$ via the following Poincar\'e-type inequality.

\begin{theorem}\label{thm_powercomp}
Let $\mu$ be a signed measure with $\supp\mu\subset \overline{B(0;R)}$ and $\int_{\mathbb{R}^d}\rd{\mu(\bx)}=\int_{\mathbb{R}^d}\bx\rd{\mu(\bx)}=0$. Then, for any $0\le \beta<\alpha<d+4$,
\begin{equation}
	\int_{\mathbb{R}^d}|\xi|^{-\alpha}|\hat{\mu}(\xi)|^2\rd{\xi} \le C_{\alpha,\beta} R^{\alpha-\beta}\int_{\mathbb{R}^d}|\xi|^{-\beta}|\hat{\mu}(\xi)|^2\rd{\xi}\,,
\end{equation}
for some $C_{\alpha,\beta}>0$.
\end{theorem}

\begin{proof}
By rescaling, we may assume $R=1$. In fact, for a general $\mu$ with $\int_{\mathbb{R}^d}\rd{\mu(\bx)}=\int_{\mathbb{R}^d}\bx\rd{\mu(\bx)}=0$ and $\supp\mu\subset \overline{B(0;R)}$, take $\tilde{\mu}(\bx) = \mu(R\bx)$, then $\supp\tilde{\mu}\subset \overline{B(0;1)}$, with $\hat{\tilde{\mu}}(\xi) = R^{-d}\hat{\mu}(\frac{\xi}{R})$,
\begin{equation}
	\int_{\mathbb{R}^d}|\xi|^{-\alpha}|\hat{\mu}(\xi)|^2\rd{\xi} = R^{d+\alpha}\int_{\mathbb{R}^d}|\xi|^{-\alpha}|\hat{\tilde{\mu}}(\xi)|^2\rd{\xi}\,,
\end{equation}
and similar for the $\beta$ term. Then the conclusion for $\mu$ may follow from that for $\tilde{\mu}$.

Let $\phi$ be a mollifier (see the last paragraph of Section \ref{sec_intro}). Then $\hat{\phi}$ is a radial smooth function with rapid decay, with $\hat{\phi}(0) = 1$. Using the fact that $(\nabla \hat{\phi})(\xi) =  \int_{\overline{B(0;1)}}-2\pi i \bx \phi(\bx)e^{-2\pi i \bx\cdot\xi}\rd{\bx}$, we see that $\|\nabla \hat{\phi}\|_{L^\infty} \le 2\pi$, and thus
\begin{equation}\label{Phihatlow}
	\hat{\phi}(\xi) \ge \frac{1}{2},\quad\forall |\xi| \le \frac{1}{4\pi}\,.
\end{equation}

Define
\begin{equation}
	\nu = \phi*\mu\,.
\end{equation}
Then $\nu$ is a smooth function supported on $\overline{B(0;2)}$, with $\int_{\mathbb{R}^d}\nu(\bx)\rd{\bx}=\int_{\mathbb{R}^d}\bx\nu(\bx)\rd{\bx}=0$. It follows that
\begin{equation}
	\hat{\nu}(0) = \nabla \hat{\nu}(0) = 0\,,
\end{equation}
and
\begin{equation}
	\left|\frac{\partial^2}{\partial \xi_j \partial \xi_k}\hat{\nu}(\xi)\right| = \left|\int_{\overline{B(0;2)}}-4\pi^2 x_jx_k \nu(\bx)e^{-2\pi i \bx\cdot\xi}\rd{\bx}\right| \le 4\pi^2 \cdot 4 \|\nu\|_{L^1} \le 16\pi^2 |B(0;2)|^{1/2}\|\nu\|_{L^2}\,,
\end{equation}
for any $\xi\in\mathbb{R}^d$. It follows that the maximal possible eigenvalue (in absolute value) of $\nabla^2 \hat{\nu}(\xi)$ is no more than $16\pi^2 |B(0;2)|^{1/2}\|\nu\|_{L^2}\cdot d$, and thus
\begin{equation}\label{hatnu}
	|\hat{\nu}(\xi)| \le 8\pi^2 d |B(0;2)|^{1/2}\|\nu\|_{L^2} \cdot |\xi|^2\,,
\end{equation}
for any $\xi\in\mathbb{R}^d$.

Take $\lambda>0$ to be determined, and we consider the quantity
\begin{equation}
	A = \int_{\mathbb{R}^d}(|\xi|^{-\beta}-\lambda|\xi|^{-\alpha}) |\hat{\mu}(\xi)|^2\rd{\xi}\,,
\end{equation}
and aim to prove $A\ge 0$ for properly chosen $\lambda$. We first notice that
\begin{equation}\label{rlam}
	|\xi|^{-\beta}-\lambda|\xi|^{-\alpha} \ge \frac{1}{2}|\xi|^{-\beta},\quad \forall |\xi| \ge r_\lambda\,,
\end{equation}
where 
\begin{equation}
	r_\lambda = (2\lambda)^{\frac{1}{\alpha-\beta}}\,.
\end{equation}
We will take $\lambda$ satisfying
\begin{equation}
	\lambda \le \frac{1}{2(4\pi)^{\alpha-\beta}}\,,
\end{equation}
so that 
\begin{equation}\label{rlam4pi}
	r_\lambda \le \frac{1}{4\pi}\,.
\end{equation}

For small $\xi$ values, we first estimate the negative quantity
\begin{equation}
	-\lambda \int_{B(0;r_\lambda)}|\xi|^{-\alpha}|\hat{\mu}(\xi)|^2\rd{\xi}\,.
\end{equation}
Here, using \eqref{Phihatlow}\eqref{rlam4pi}\eqref{hatnu}, the last integral can be estimated by
\begin{equation}\begin{split}
		\int_{B(0;r_\lambda)}|\xi|^{-\alpha}|\hat{\mu}(\xi)|^2\rd{\xi} \le & 4\int_{B(0;r_\lambda)}|\xi|^{-\alpha}|\hat{\mu}(\xi)|^2|\hat{\phi}(\xi)|^2\rd{\xi} = 4\int_{B(0;r_\lambda)}|\xi|^{-\alpha}|\hat{\nu}(\xi)|^2\rd{\xi} \\
		\le & 4(8\pi^2 d |B(0;2)|^{1/2}\|\nu\|_{L^2})^2\int_{B(0;r_\lambda)}|\xi|^{-\alpha} |\xi|^4\rd{\xi} \\
		= & 4(8\pi^2 d)^2 |B(0;2)| \|\nu\|_{L^2}^2 \cdot |\partial B(0;1)| \frac{1}{4-\alpha+d}r_\lambda^{4-\alpha+d}\,.
\end{split}\end{equation}
Therefore we conclude
\begin{equation}\label{Aest1}
	-\lambda \int_{B(0;r_\lambda)}|\xi|^{-\alpha}|\hat{\mu}(\xi)|^2\rd{\xi} \ge -C_1 \lambda r_\lambda^{4-\alpha+d} \|\nu\|_{L^2}^2 \,,
\end{equation}
where
\begin{equation}\label{eqC1}
	C_1 = 4(8\pi^2 d)^2 |B(0;2)|\cdot |\partial B(0;1)| \frac{1}{4-\alpha+d}\,.
\end{equation}

For large $\xi$ values, together with positive contribution from small $\xi$, we estimate as
\begin{equation}
	\int_{B(0;r_\lambda)} |\xi|^{-\beta}|\hat{\mu}(\xi)|^2\rd{\xi} +  \int_{B(0;r_\lambda)^c}(|\xi|^{-\beta}-\lambda|\xi|^{-\alpha}) |\hat{\mu}(\xi)|^2\rd{\xi} \ge \frac{1}{2}\int_{\mathbb{R}^d}|\xi|^{-\beta} |\hat{\mu}(\xi)|^2\rd{\xi}\,,
\end{equation}
by \eqref{rlam}. Notice that $|\xi|^{2+\beta}|\hat{\phi}(\xi)|^2$ is a continuous radial function which decays at infinity, and thus it is bounded, and then
\begin{equation}\label{eqC20}
	|\xi|^{-\beta} \ge \frac{2}{C_2} |\xi|^2|\hat{\phi}(\xi)|^2,\quad \forall \xi\ne 0\,,
\end{equation}
where 
\begin{equation}\label{eqC2}
	C_2 = 2\sup_{\xi\in\mathbb{R}^d} |\xi|^{2+\beta}|\hat{\phi}(\xi)|^2\,.
\end{equation}
It follows that 
\begin{equation}\begin{split}
		 \int_{B(0;r_\lambda)} & |\xi|^{-\beta}|\hat{\mu}(\xi)|^2\rd{\xi} +  \int_{B(0;r_\lambda)^c}(|\xi|^{-\beta}-\lambda|\xi|^{-\alpha}) |\hat{\mu}(\xi)|^2\rd{\xi} \\
		\ge & \frac{1}{C_2}\int_{\mathbb{R}^d} |\xi|^2|\hat{\phi}(\xi)|^2 |\hat{\mu}(\xi)|^2\rd{\xi} = \frac{1}{C_2}\int_{\mathbb{R}^d} |\xi|^2 |\hat{\nu}(\xi)|^2\rd{\xi}\,.
\end{split}\end{equation}
Combining with \eqref{Aest1} we get
\begin{equation}
	A \ge \frac{1}{C_2}\int_{\mathbb{R}^d} |\xi|^2 |\hat{\nu}(\xi)|^2\rd{\xi} -C_1 \lambda r_\lambda^{4-\alpha+d} \|\nu\|_{L^2}^2\,.
\end{equation}

Finally we apply the Heisenberg uncertainty principle (c.f. \cite[Corollary 2.8]{FS97}) to $\nu$, which gives
\begin{equation}
	\int_{\mathbb{R}^d} |\bx|^2 |\nu(\bx)|^2\rd{\bx} \cdot \int_{\mathbb{R}^d} |\xi|^2 |\hat{\nu}(\xi)|^2\rd{\xi} \ge \frac{d^2}{16\pi^2}\|\nu\|_{L^2}^4\,.
\end{equation}
Since $\nu$ is supported on $\overline{B(0;2)}$, we see that
\begin{equation}
	\int_{\mathbb{R}^d} |\bx|^2 |\nu(\bx)|^2\rd{\bx} \le 4\|\nu\|_{L^2}^2\,,
\end{equation}
which implies
\begin{equation}\label{eqHei}
	\int_{\mathbb{R}^d} |\xi|^2 |\hat{\nu}(\xi)|^2\rd{\xi} \ge \frac{d^2}{64\pi^2}\|\nu\|_{L^2}^2\,.
\end{equation}
Therefore we get
\begin{equation}
	A \ge \frac{d^2}{64\pi^2 C_2}\|\nu\|_{L^2}^2 -C_1 \lambda r_\lambda^{4-\alpha+d} \|\nu\|_{L^2}^2 = \frac{d^2}{64\pi^2 C_2}\|\nu\|_{L^2}^2\Big(1- 64\pi^2 d^{-2} C_1C_2 \cdot 2^{\frac{4-\alpha+d}{\alpha-\beta}}\lambda^{1+\frac{4-\alpha+d}{\alpha-\beta}}\Big)\,,
\end{equation}
where the power on $\lambda$ is $1+\frac{4-\alpha+d}{\alpha-\beta} = \frac{4-\beta+d}{\alpha-\beta}>0$. By taking
\begin{equation}\label{lamfinal}
	\lambda = \min\left\{\frac{1}{2(4\pi)^{\alpha-\beta}} , (64\pi^2 d^{-2}  C_1C_2 )^{-\frac{\alpha-\beta}{4-\beta+d}}\cdot  2^{-\frac{4-\alpha+d}{4-\beta+d}}\right\}\,,
\end{equation}
we get $A\ge 0$, which gives the conclusion with $C_{\alpha,\beta}=\frac{1}{\lambda}$.

\end{proof}

\begin{remark}[Dependence of $C_{\alpha,\beta}$ on $\beta$ and $\alpha$]\label{rmk_powercomp}
In \eqref{lamfinal} in the above proof, notice that $C_2$, given by \eqref{eqC2}, is uniformly bounded for any $\beta,\alpha$ in the proposed range; $C_1$, given by \eqref{eqC1}, can be large when $\alpha$ is close to $d+4$; the exponents $0<\frac{\alpha-\beta}{4-\beta+d} < 1$, $0<\frac{4-\alpha+d}{4-\beta+d}<1$. Therefore $C_{\alpha,\beta}=\frac{1}{\lambda}$ can get large only when $\alpha$ is close to $d+4$, during which $C_1\sim \frac{1}{d+4-\alpha}$, $\frac{\alpha-\beta}{4-\beta+d}\approx 1$ (for fixed $\beta$) and thus $C_{\alpha,\beta}\sim \frac{1}{d+4-\alpha}$. Other than this issue, $C_{\alpha,\beta}$ is uniformly bounded if $\alpha$ is away from $d+4$.
\end{remark}


Then we give the logarithmic analogue of Theorem \ref{thm_powercomp}.

\begin{theorem}\label{thm_powercomplog}
	Let $\mu$ be a signed measure with $\supp\mu\subset \overline{B(0;R)}$ and $\int_{\mathbb{R}^d}\rd{\mu(\bx)}=\int_{\mathbb{R}^d}\bx\rd{\mu(\bx)}=0$. Then, for any $0\le \beta<d+4$,
	\begin{equation}
		\int_{\mathbb{R}^d}|\xi|^{-\beta}\ln|\xi|\cdot|\hat{\mu}(\xi)|^2\rd{\xi} + (C_\beta+\ln R)\int_{\mathbb{R}^d}|\xi|^{-\beta}|\hat{\mu}(\xi)|^2\rd{\xi} \ge 0\,,
	\end{equation}
	for some $C_\beta>0$.
\end{theorem}

\begin{proof}
	By rescaling, we may assume $R=1$. In fact, for a general $\supp\mu\subset \overline{B(0;R)}$, take $\tilde{\mu}(\bx) = \mu(R\bx)$, then $\supp\tilde{\mu}\subset \overline{B(0;1)}$, with $\hat{\tilde{\mu}}(\xi) = R^{-d}\hat{\mu}(\frac{\xi}{R})$,
	\begin{equation}
		\int_{\mathbb{R}^d}|\xi|^{-\beta}|\hat{\mu}(\xi)|^2\rd{\xi} = R^{d+\beta}\int_{\mathbb{R}^d}|\xi|^{-\beta}|\hat{\tilde{\mu}}(\xi)|^2\rd{\xi}\,,
	\end{equation}
	and
	\begin{equation}
		\int_{\mathbb{R}^d}|\xi|^{-\beta}\ln|\xi|\cdot|\hat{\mu}(\xi)|^2\rd{\xi} + \ln R \int_{\mathbb{R}^d}|\xi|^{-\beta}|\hat{\mu}(\xi)|^2\rd{\xi}= R^{d+\beta}\int_{\mathbb{R}^d}|\xi|^{-\beta}\ln|\xi|\cdot|\hat{\tilde{\mu}}(\xi)|^2\rd{\xi}\,.
	\end{equation}
	Then the conclusion for $\mu$ may follow from that for $\tilde{\mu}$.
	
	Define $\phi$, $\nu$ in the same way as the proof of Theorem \ref{thm_powercomp}. Take $\lambda>0$ to be determined, and we consider the quantity
	\begin{equation}
		A = \int_{\mathbb{R}^d}|\xi|^{-\beta}(\ln|\xi|+\lambda) |\hat{\mu}(\xi)|^2\rd{\xi}\,,
	\end{equation}
	and aim to prove $A\ge 0$ for properly chosen $\lambda\ge 2$. We first notice that
	\begin{equation}\label{rlam1}
		\ln|\xi|+\lambda \ge \frac{\lambda}{2},\quad \forall |\xi| \ge r_\lambda\,,
	\end{equation}
	where 
	\begin{equation}
		r_\lambda = e^{-\lambda/2}\,.
	\end{equation}
	We will take $\lambda$ satisfying
	\begin{equation}
		\lambda \ge 2\ln(4\pi)\,,
	\end{equation}
	so that 
	\begin{equation}\label{rlam4pi1}
		r_\lambda \le \frac{1}{4\pi}\,.
	\end{equation}
	
	For small $\xi$ values, we first estimate the negative quantity
	\begin{equation}
		\int_{B(0;r_\lambda)}|\xi|^{-\beta}\ln|\xi|\cdot|\hat{\mu}(\xi)|^2\rd{\xi}\,.
	\end{equation}
	For any $|\xi|< r_\lambda=e^{-\lambda/2}$, we have $\ln|\xi| < -\frac{\lambda}{2} < 0$, and thus the above integrand is non-positive. Using \eqref{Phihatlow}\eqref{rlam4pi1}\eqref{hatnu}, the last integral can be estimated by
	\begin{equation}\begin{split}
			&\int_{B(0;r_\lambda)}|\xi|^{-\beta}(-\ln|\xi|)\cdot|\hat{\mu}(\xi)|^2\rd{\xi}\\
			\le & 4\int_{B(0;r_\lambda)}|\xi|^{-\beta}(-\ln|\xi|)\cdot|\hat{\mu}(\xi)|^2|\hat{\phi}(\xi)|^2\rd{\xi} = 4\int_{B(0;r_\lambda)}|\xi|^{-\beta}(-\ln|\xi|)\cdot|\hat{\nu}(\xi)|^2\rd{\xi} \\
			\le & 4(8\pi^2 d |B(0;2)|^{1/2}\|\nu\|_{L^2})^2\int_{B(0;r_\lambda)}|\xi|^{-\beta}(-\ln|\xi|)\cdot |\xi|^4\rd{\xi} \\
			= & 4(8\pi^2 d)^2 |B(0;2)| \|\nu\|_{L^2}^2 \cdot |\partial B(0;1)| \Big(\frac{1}{4-\beta+d}r_\lambda^{4-\beta+d}(-\ln r_\lambda) + \frac{1}{(4-\beta+d)^2}r_\lambda^{4-\beta+d}\Big) \\
			\le & 4(8\pi^2 d)^2 |B(0;2)| \|\nu\|_{L^2}^2 \cdot |\partial B(0;1)| \Big(\frac{1}{4-\beta+d}+\frac{1}{(4-\beta+d)^2\ln(4\pi)}\Big)r_\lambda^{4-\beta+d}(-\ln r_\lambda)\,. \\
	\end{split}\end{equation}
	Therefore we conclude
	\begin{equation}\label{Aest1log}
		\int_{B(0;r_\lambda)}|\xi|^{-\beta}\ln|\xi|\cdot|\hat{\mu}(\xi)|^2\rd{\xi} \ge -C_1   r_\lambda^{4-\beta+d}(-\ln r_\lambda) \|\nu\|_{L^2}^2 \,,
	\end{equation}
	where
	\begin{equation}\label{eqC1log}
		C_1 = 4(8\pi^2 d)^2 |B(0;2)|\cdot |\partial B(0;1)| \Big(\frac{1}{4-\beta+d}+\frac{1}{(4-\beta+d)^2\ln(4\pi)}\Big)\,.
	\end{equation}
	
	For large $\xi$ values, together with positive contribution from small $\xi$, we estimate as
	\begin{equation}\begin{split}
		\int_{B(0;r_\lambda)} & \lambda|\xi|^{-\beta}|\hat{\mu}(\xi)|^2\rd{\xi} +  \int_{B(0;r_\lambda)^c}|\xi|^{-\beta}(\ln|\xi|+\lambda) |\hat{\mu}(\xi)|^2\rd{\xi} \\ \ge & \frac{\lambda}{2}\int_{\mathbb{R}^d}|\xi|^{-\beta} |\hat{\mu}(\xi)|^2\rd{\xi} 
		\ge  \frac{\lambda}{C_2}\int_{\mathbb{R}^d} |\xi|^2|\hat{\phi}(\xi)|^2 |\hat{\mu}(\xi)|^2\rd{\xi} = \frac{\lambda}{C_2}\int_{\mathbb{R}^d} |\xi|^2 |\hat{\nu}(\xi)|^2\rd{\xi}\,,
	\end{split}\end{equation}
	by \eqref{rlam1} and \eqref{eqC20}, where $C_2$ is given by \eqref{eqC2}. Combining with \eqref{Aest1log} we get
	\begin{equation}
		A \ge \frac{\lambda}{C_2}\int_{\mathbb{R}^d} |\xi|^2 |\hat{\nu}(\xi)|^2\rd{\xi} -C_1  r_\lambda^{4-\beta+d}(-\ln r_\lambda) \|\nu\|_{L^2}^2\,.
	\end{equation}
	
	Finally we use \eqref{eqHei} from the Heisenberg uncertainty principle and get
	\begin{equation}\begin{split}
			A \ge & \frac{\lambda}{64\pi^2d^{-2} C_2}\|\nu\|_{L^2}^2 -C_1  r_\lambda^{4-\beta+d}(-\ln r_\lambda) \|\nu\|_{L^2}^2 \\
			= & \frac{\lambda}{64\pi^2d^{-2} C_2}\|\nu\|_{L^2}^2\Big(1- 64\pi^2d^{-2}  C_1C_2 \cdot \lambda^{-1}e^{-\lambda(4-\beta+d)/2}\frac{\lambda}{2}\Big) \\
			= & \frac{\lambda}{64\pi^2d^{-2} C_2}\|\nu\|_{L^2}^2\Big(1- 32\pi^2d^{-2}  C_1C_2 \cdot e^{-\lambda(4-\beta+d)/2}\Big) \,,
	\end{split}\end{equation}
	where in the last exponential one has $4-\beta+d>0$. By taking
	\begin{equation}\label{lamfinal1}
		\lambda = \max\left\{2\ln(4\pi) , \frac{2}{4-\beta+d}\ln(32\pi^2d^{-2}  C_1C_2) \right\}\,,
	\end{equation}
	we get $A\ge 0$, which gives the conclusion with $C_{\beta}=\lambda$.
	
\end{proof}

\begin{remark}[Dependence of $C_{\beta}$ on  $\beta$]\label{rmk_powercomplog}
	In \eqref{lamfinal1} in the above proof, notice that $C_2$, given by \eqref{eqC2}, is uniformly bounded for any $\beta$ in the proposed range; $C_1$, given by \eqref{eqC1log}, and the fraction $\frac{2}{4-\beta+d}$  can be large when $\beta$ is close to $d+4$. Therefore $C_{\beta}=\lambda$ can get large only when $\beta$ is close to $d+4$, during which $C_1\sim \frac{1}{(d+4-\beta)^2}$, and thus $C_{\beta}\sim \frac{1}{d+4-\beta}(-\ln(d+4-\beta))$. Other than this issue, $C_{\beta}$ is uniformly bounded if $\beta$ is away from $d+4$.
\end{remark}

\section{Proof of Theorems \ref{thm_main1} and \ref{thm_main3}}\label{sec_proof1}

\begin{proof}[Proof of Theorem \ref{thm_main1}]
Take $a_0 = \max\{1,\frac{b+2}{2}\}$. Then Theorem \ref{thm_existab} shows that $R_*(b,a_0)$ is an upper bound of minimizer size for any $W_{a,b}$ with $a\ge a_0$. 

We claim that for $a_-(b)\in [a_0,2)$ sufficiently close to 2, we have 
\begin{equation}\label{claim_RLIC}
	R_{\textnormal{LIC}}[W_{a,b}] \ge R_*(b,a_0) +1 =: R,\quad \forall a_-(b) \le a < 2\,,
\end{equation}
which would imply the conclusion due to Lemma \ref{lem_LICR}.

By Theorems \ref{thm_RLICab} and \ref{thm_powerF}, it suffices to prove that for any nontrivial signed measure $\mu$ supported on $\overline{B(0;R)}$ with $E_{a,b}[|\mu|]<\infty$ and $\int_{\mathbb{R}^d}\rd{\mu(\bx)}=\int_{\mathbb{R}^d}\bx\rd{\mu(\bx)}=0$, there holds
\begin{equation}\label{EabFou}
	\int_{\mathbb{R}^d\backslash \{0\}} \hat{W}_{a,b}(\xi)|\hat{\mu}(\xi)|^2\rd{\xi} = \int_{\mathbb{R}^d}   \Big(\cpf(-b) |\xi|^{-d-b}   -\cpf(-a) |\xi|^{-d-a}\Big)|\hat{\mu}(\xi)|^2\rd{\xi} > 0\,.
\end{equation}

Notice that $\cpf(s)>0$ for any $-2<s<d$, with $\cpf(-2)=0$ and
\begin{equation}
	\cpf(s)  \approx \cpf'(-2)(s+2),\quad \cpf'(-2)=\pi^{-2-\frac{d}{2}}\frac{\Gamma(\frac{d+2}{2})}{4}>0\,,
\end{equation}
for $s$ close to $-2$, to the leading order (c.f. Remark \ref{rmk_cpf}). This implies that $\cpf$ is an increasing function on $[-2,-2+\epsilon]$ for some $\epsilon>0$.  Then we apply Theorem \ref{thm_powercomp} to $\alpha=d+a$ and $\beta=d+b$ to get
\begin{equation}
	\int_{\mathbb{R}^d}|\xi|^{-d-a}|\hat{\mu}(\xi)|^2\rd{\xi} \le C_*\int_{\mathbb{R}^d}|\xi|^{-d-b}|\hat{\mu}(\xi)|^2\rd{\xi}\,.
\end{equation}
Here $C_* = \sup_{a\in [a_0,2)}(C_{d+a,d+b}R^{a-b}) >0$ is a uniform constant since $\alpha=d+a$ is bounded above by $d+2$ and thus away from $d+4$ (c.f. Remark \ref{rmk_powercomp}). Then we may take $a_-(b)\in [a_0,2)$ sufficiently close to 2 so that $2-a_-(b)<\epsilon$ and
{\begin{equation}
	\cpf(-a_-(b)) \le \frac{1}{2C_*}\cpf(-b)\,.
\end{equation}
It follows that $\cpf(-a) \le \cpf(-a_-(b)) \le \frac{1}{2C_*}\cpf(-b)$} for any $a\in [a_-(b),2)$, and we see that 
\begin{equation}
	\int_{\mathbb{R}^d\backslash \{0\}} \hat{W}_{a,b}(\xi)|\hat{\mu}(\xi)|^2\rd{\xi} > 0\,,
\end{equation}
which finishes the proof.

\end{proof}

\begin{remark}\label{rmk_sharp}
	Let $C_{\alpha,\beta}^\sharp>0$ be the sharp constant in Theorem \ref{thm_powercomp}. Then \eqref{EabFou} shows that
	\begin{equation}
		R_{\textnormal{LIC}}[W_{a,b}] = \Big(\frac{\cpf(-b)}{C_{d+a,d+b}^\sharp \cpf(-a)} \Big)^{1/(a-b)}\,,
	\end{equation}
	for any $-d<b<a<2$.
\end{remark}

\begin{proof}[Proof of Theorem \ref{thm_main3}]
	Theorem \ref{thm_existablog} gives $R_{*,\ln}(1)$ as an upper bound of minimizer size for any $W_{b,\ln}$ with $b\ge 1$. 
	
	We claim that for $b_*\in [1,2)$ sufficiently close to 2, we have 
	\begin{equation}\label{claim_RLIClog}
		R_{\textnormal{LIC}}[W_{b,\ln}] \ge R_{*,\ln}(1) + 1 =: R,\quad \forall b_* \le b < 2\,,
	\end{equation}
	which would imply the conclusion due to Lemma \ref{lem_LICR}.
	
	By Theorems \ref{thm_RLICab} and \ref{thm_powerF}, it suffices to prove that for any nontrivial signed measure $\mu$ supported on $\overline{B(0;R)}$ with $E_{a,b}[|\mu|]<\infty$ and $\int_{\mathbb{R}^d}\rd{\mu(\bx)}=\int_{\mathbb{R}^d}\bx\rd{\mu(\bx)}=0$, there holds
	\begin{equation}\label{EabFoulog}
		\int_{\mathbb{R}^d\backslash \{0\}} \hat{W}_{b,\ln}(\xi)|\hat{\mu}(\xi)|^2\rd{\xi} = \int_{\mathbb{R}^d}   \Big(\tcpf(-b) |\xi|^{-d-b}   +\cpf(-b) |\xi|^{-d-b}\ln|\xi|\Big)|\hat{\mu}(\xi)|^2\rd{\xi} > 0\,.
	\end{equation}
	The properties of $\cpf(s)$ near $s=-2$ were analyzed in the proof of Theorem \ref{thm_main1}. We recall that $\tcpf(-2)=\cpf'(-2)>0$ from \eqref{eqtcpf2}. 	Therefore one can take $\epsilon>0$ such that $\cpf$ is increasing on $[-2,-2+\epsilon]$ (as discussed in the proof of Theorem \ref{thm_main1}), and $\tcpf(s) \ge \frac{1}{2}\cpf'(-2) > 0$ for $s\in [-2,-2+\epsilon]$. 
	
	Then we apply Theorem \ref{thm_powercomplog} to $\beta=d+b$ to get
	\begin{equation}
		\int_{\mathbb{R}^d}|\xi|^{-d-b}\ln|\xi|\cdot|\hat{\mu}(\xi)|^2\rd{\xi} + 	C_*\int_{\mathbb{R}^d}|\xi|^{-d-b}|\hat{\mu}(\xi)|^2\rd{\xi} \ge 0\,,
	\end{equation}
	for any $1\le b < 2$. Here $C_*=\sup_{1\le b < 2}C_{d+b} + \ln R>0$ is a uniform constant since $\beta$ is bounded above by $d+2$ and thus away from $d+4$ (c.f. Remark \ref{rmk_powercomplog}). Then we may take $b_*\in [1,2)$ sufficiently close to 2 so that $2-b_*<\epsilon$ and
	\begin{equation}
		\cpf(-b_*) \le \frac{1}{4C_*}\cpf'(-2)\,.
	\end{equation}
	It follows that 
	\begin{equation}
		\cpf(-b) \le \cpf(-b_*) \le \frac{1}{4C_*}\cpf'(-2) \le \frac{1}{2C_*}\tcpf(-b)\,,
	\end{equation} 
	for any $b\in [b_*,2)$, and we see that 
	\begin{equation}
		\int_{\mathbb{R}^d\backslash \{0\}} \hat{W}_{b,\ln}(\xi)|\hat{\mu}(\xi)|^2\rd{\xi} > 0\,,
	\end{equation}
	which finishes the proof.
	
\end{proof}

\section{Proof of Theorem \ref{thm_main2}}\label{sec_proof2}

Compared to the proof of Theorem \ref{thm_main1}, the main difficulty to handle $|\bx|^a$ with $a>4$ is that it is not Fourier representable. We will first truncate this attractive potential and then estimate its Fourier transform. To gain smallness in this estimate when $a$ is close to 4, we need to subtract $|\bx|^4$ (which we know to have the LIC) from the potential. Such smallness will play a crucial role in the proof of Theorem \ref{thm_main2}, similar to the smallness of $\cpf(-a)$ in \eqref{EabFou} in the proof of Theorem \ref{thm_main1}.

\begin{lemma}\label{lem_a4}
	Let $a>4$, and $\Phi$ be a truncation function (see the last paragraph of Section \ref{sec_intro}). Then there exists $\eta_a>0$ such that
	\begin{equation}
		\Big|\cF[(|\bx|^a-|\bx|^4)\Phi(\bx)](\xi)\Big| \le \eta_a (1+|\xi|)^{-d-a}\,.
	\end{equation}
	Furthermore, $\eta_a = O(a-4)$ when $a$ is close to 4.
\end{lemma}

\begin{proof}
	The function $(|\bx|^a-|\bx|^4)\Phi(\bx)$ is compactly supported and continuous, and thus its Fourier transform is given by
	\begin{equation}\label{Fxa4}
		\cF[(|\bx|^a-|\bx|^4)\Phi(\bx)](\xi)  = \int_{\mathbb{R}^d}(|\bx|^a-|\bx|^4)\Phi(\bx) e^{-2\pi i \bx\cdot\xi} \rd{\bx}\,.
	\end{equation}
	
	We first assume $d\ge 2$. Let
	\begin{equation}
		k = 1+\lfloor \frac{a}{2} \rfloor\,.
	\end{equation}	
	Multiplying by $(-4\pi^2|\xi|^2)^k$ in \eqref{Fxa4} and integrating by parts, we get
	\begin{equation}\begin{split}
		(-4\pi^2|\xi|^2)^k\cF[(|\bx|^a-|\bx|^4)\Phi(\bx)](\xi)  = & \int_{\mathbb{R}^d}(|\bx|^a-|\bx|^4)\Phi(\bx) 	\Delta^k e^{-2\pi i \bx\cdot\xi} 
		 \rd{\bx} \\
		  = & \int_{\mathbb{R}^d}	\Delta^k\big((|\bx|^a-|\bx|^4)\Phi(\bx)\big)  e^{-2\pi i \bx\cdot\xi} 
		 \rd{\bx}\,. \\
	\end{split}\end{equation}
	Here the integration by parts is legal because the integrand is compactly supported and $\Delta^k\big((|\bx|^a-|\bx|^4)\Phi(\bx)\big)$ is  locally integrable in $d\ge 2$, by the choice of $k$. Since $\Phi$ is supported on $\overline{B(0;2)}$, the last integrand is zero for any $|\bx|>2$. Then we estimate its size by separating it into
	\begin{equation}\label{Deltaa4}\begin{split}
		\int_{\mathbb{R}^d} &	\Delta^k\big((|\bx|^a-|\bx|^4)\Phi(\bx)\big)  e^{-2\pi i \bx\cdot\xi} 
		\rd{\bx} \\
		= & \int_{|\bx|\le 1}	\Delta^k\big((|\bx|^a-|\bx|^4)\Phi(\bx)\big)\cdot\Phi(2\bx)  e^{-2\pi i \bx\cdot\xi}\rd{\bx} \\
		& + \int_{\frac{1}{2}<|\bx|\le 2}	\Delta^k\big((|\bx|^a-|\bx|^4)\Phi(\bx)\big)\cdot(1-\Phi(2\bx))  e^{-2\pi i \bx\cdot\xi}
		\rd{\bx}\,.
	\end{split}\end{equation}
	
	First, since $\Phi=1$ on $B(0;1)$ and $k\ge 3$,
	\begin{equation}\begin{split}
		\int_{|\bx|\le 1} &	\Delta^k\big((|\bx|^a-|\bx|^4)\Phi(\bx)\big)\cdot\Phi(2\bx)  e^{-2\pi i \bx\cdot\xi}\rd{\bx} =  \int_{|\bx|\le 1}	\Delta^k(|\bx|^a)\cdot\Phi(2\bx)  e^{-2\pi i \bx\cdot\xi} 
		\rd{\bx}\,.
	\end{split}\end{equation}
	If $a>4$ is an even integer, then $a=2(k-1)$ and the last integrand is zero. Otherwise, we have $-2<a-2k<0$, and then
	\begin{equation}\begin{split}
		\int_{|\bx|\le 1} &	\Delta^k\big((|\bx|^a-|\bx|^4)\Phi(\bx)\big)\cdot\Phi(2\bx)  e^{-2\pi i \bx\cdot\xi}\rd{\bx} \\
		= & \prod_{j=0}^{k-1} (a-2j)(a+d-2-2j) \cdot \int_{|\bx|\le 1}	|\bx|^{a-2k}\Phi(2\bx)  e^{-2\pi i \bx\cdot\xi} 
		\rd{\bx} \\
		= &  \prod_{j=0}^{k-1} (a-2j)(a+d-2-2j) \cdot	\cF[|\bx|^{a-2k}\Phi(2\bx)](\xi)\,. \\
	\end{split}\end{equation}
	The Fourier transform of $|\bx|^{a-2k}$ is a locally integrable function given by Theorem \ref{thm_powerF}, and then
	\begin{equation}
		\cF[|\bx|^{a-2k}\Phi(2\bx)] = (2k-a)\cpf(2k-a)2^{-d}\cdot |\xi|^{-d+2k-a}*\hat{\Phi}(\frac{\xi}{2})\,.
	\end{equation}
	
	{Then we claim that there exists $C_1>0$, uniform in $a$ for $-2<a-2k<0$, such that
		\begin{equation}\label{claim_2ka}
			\Big|(2k-a)\big(|\cdot|^{-d+2k-a}*\hat{\Phi}(\frac{\cdot}{2})\big)(\xi)\Big| \le C_1(1+|\xi|)^{-d+2k-a}\,.
		\end{equation}
	In fact, we first apply Lemma \ref{lem_psi} with $\psi_1 = \hat{\Phi}(\frac{\cdot}{2})$, $\psi_2=(2k-a)|\cdot|^{-d+2k-a}$, $M=-d+2k-a$, $R=1$ to obtain the above inequality for $|\xi|>3$.
	By tracing through the proof of Lemma \ref{lem_psi}, the constant $C_1$ only depends on $\psi_2$ through the implied constant in \eqref{lem_psi_0} and $\int_{B(0;1)}|\psi_2(\bx)|\rd{\bx}$, which are both uniform in $a$ for our choice of $\psi_2$. Therefore $C_1$ can be chosen uniformly in $a$.
	
	To see this inequality for $|\xi|\le 3$, we write its LHS as
	\begin{equation}
		|(\psi_1*\psi_2)(\xi)| \le |(\psi_1*(\psi_2\chi_{B(0;1)})(\xi)| + |(\psi_1*(\psi_2\chi_{B(0;1)^c})(\xi)|\,.
	\end{equation}
	By Young's inequality, we have
	\begin{equation}
		|(\psi_1*(\psi_2\chi_{B(0;1)})(\xi)| \le \|\psi_1\|_{L^\infty(\mathbb{R}^d)} \cdot \|\psi_2\|_{L^1(B(0;1))} \le C\,,
	\end{equation}
	and
	\begin{equation}
		|(\psi_1*(\psi_2\chi_{B(0;1)^c})(\xi)| \le \|\psi_1\|_{L^1(\mathbb{R}^d)} 	\cdot \|\psi_2\|_{L^\infty(B(0;1)^c)} \le C\,,
	\end{equation}
	uniform in $a$. Therefore the claim \eqref{claim_2ka} is proved.
	}


	 To analyze the behavior of the $\cpf$ part,
	\begin{itemize}
		\item If $d\ge 3$, then $2k-a\in (0,2)$ is strictly inside $\{s\in\mathbb{C}:\Re(s) < d\}$ where $\cpf$ is holomorphic, and thus $\cpf(2k-a)$ is uniformly bounded.
		Therefore 
		\begin{equation}\begin{split}
				\left|\int_{|\bx|\le 1}	\Delta^k\big((|\bx|^a-|\bx|^4)\Phi(\bx)\big)\cdot\Phi(2\bx)  e^{-	2\pi i \bx\cdot\xi}\rd{\bx} \right| \le \eta_{a,1}(1+|\xi|)^{-d+2k-a}\,,
		\end{split}\end{equation}
		where 
		\begin{equation}\label{etaa1_d3}
			\eta_{a,1} = \prod_{j=0}^{k-1} (a-2j)(a+d-2-2j)\cdot C_1 \cdot \sup_{	0<s<2}\cpf(s)\,.
		\end{equation}
		\item If $d=2$, then $2k-a\in (0,2)$ is inside $\{s\in\mathbb{C}:\Re(s) < d\}$ but can come close to the point 2 where $\cpf$ has a simple pole (coming from the expression $\Gamma(\frac{d-s}{2})$ in its definition). Therefore, using $\cpf(s)\lesssim (2-s)^{-1}$, we get
		\begin{equation}\begin{split}
				\left|\int_{|\bx|\le 1}	\Delta^k\big((|\bx|^a-|\bx|^4)\Phi(\bx)\big)\cdot\Phi(2\bx)  e^{-	2	\pi i \bx\cdot\xi}\rd{\bx} \right| \le \eta_{a,1}(1+|\xi|)^{-d+2k-a}\,,
		\end{split}\end{equation}
		where 
		\begin{equation}\label{etaa1_d2}
			\eta_{a,1} = \prod_{j=0}^{k-1} (a-2j)^2\cdot C_1 \cdot \frac{1}{2-2k+a}\sup_{	0<s	<2}(2-s)\cpf(s)\,.
		\end{equation}
	\end{itemize}
	Notice that in both cases we have
	\begin{equation}
		\eta_{a,1} = O(a-4)\,,
	\end{equation}
	when $a$ is close to 4. In fact, we have $k=3$ when $a$ is close to 4. For $d\ge 3$, the product term in \eqref{etaa1_d3} contains a factor $(a-4)$. For $d=2$, the product term in \eqref{etaa1_d2} contains a factor $(a-4)^2$ and there is another $\frac{1}{a-4}$ in \eqref{etaa1_d2}.
	
	To estimate the last integral in \eqref{Deltaa4}, we observe that $\Delta^k\big((|\bx|^a-|\bx|^4)\Phi(\bx)\big)\cdot(1-\Phi(2\bx))$ is smooth on $\mathbb{R}^d$ and supported on $\{\frac{1}{2}\le |\bx| \le 2\}$. Therefore, the last integral in \eqref{Deltaa4}, being the Fourier transform of this function, is a Schwartz function in $\xi$. Therefore there exists $\eta_{a,2}>0$ such that
	\begin{equation}\begin{split}
		\left|\int_{\frac{1}{2}<|\bx|\le 2}	\Delta^k\big((|\bx|^a-|\bx|^4)\Phi(\bx)\big)\cdot(1-\Phi(2\bx))  e^{-2\pi i \bx\cdot\xi}
		\rd{\bx}\right| \le \eta_{a,2}(1+|\xi|)^{-d+2k-a}\,.
	\end{split}\end{equation}
	To analyze the behavior of $\eta_{a,2}$ for $a$ close to 4, we first write
	\begin{equation}\begin{split}
		& \left|(1+4\pi^2|\xi|^2)^m  \int_{\frac{1}{2}<|\bx|\le 2}	\Delta^k\big((|\bx|^a-|\bx|^4)\Phi(\bx)\big)\cdot(1-\Phi(2\bx))  e^{-2\pi i \bx\cdot\xi}
		\rd{\bx}\right|\\
		= & \left|\int_{\frac{1}{2}<|\bx|\le 2}	 (1-\Delta)^m\Big[\Delta^k\big((|\bx|^a-|\bx|^4)\Phi(\bx)\big)\cdot(1-\Phi(2\bx))\Big]  e^{-2\pi i \bx\cdot\xi}
		\rd{\bx}\right|\\
		\le & \int_{\frac{1}{2}<|\bx|\le 2}	 \left|(1-\Delta)^m\Big[\Delta^k\big((|\bx|^a-|\bx|^4)\Phi(\bx)\big)\cdot(1-\Phi(2\bx))\Big]  \right|
		\rd{\bx}\,,\\
	\end{split}\end{equation}
	for any integer $m\ge 0$. The last integrand is a linear combination of products derivatives of $|\bx|^a-|\bx|^4, \Phi(\bx), 1-\Phi(2\bx)$. Each derivative of $\Phi(\bx)$ or $1-\Phi(2\bx)$ is bounded, and each derivative of $|\bx|^a-|\bx|^4$ is of order $O(a-4)$ for $\frac{1}{2}<|\bx|\le 2$. Therefore the last integral is $O(a-4)$ for any fixed $m$. By taking $m\ge \frac{d-2k+a}{2}$, we see that $\eta_{a,2}=O(a-4)$.
	
	Combining the above two parts, we see that the LHS of \eqref{Deltaa4} is bounded by $\eta_{a,0} (1+|\xi|)^{-d+2k-a}$ with $\eta_{a,0}=O(a-4)$ when $a$ is close to 4. Therefore $|\cF[(|\bx|^a-|\bx|^4)\Phi(\bx)](\xi)| \lesssim \eta_a |\xi|^{-2k}(1+|\xi|)^{-d+2k-a}$ with $\eta_a=O(a-4)$. Combined with the straightforward estimate $|\cF[(|\bx|^a-|\bx|^4)\Phi(\bx)](\xi)| \lesssim a-4$ for $|\xi|\le 1$, we get the conclusion when $d\ge 2$.
	
	Then we treat the case $d=1$ by a similar technique. Let
	\begin{equation}
		k = 1+\lfloor a \rfloor\,.
	\end{equation}	
	Multiplying by $(2\pi i \xi)^k$ in \eqref{Fxa4} and integrating by parts, we get
	\begin{equation}\begin{split}
			(2\pi i \xi)^k\cF[(|x|^a-|x|^4)\Phi(x)](\xi)  = & \int_{\mathbb{R}}(|x|^a-|x|^4)\Phi(x) 	\big(-\frac{\rd}{\rd x}\big)^k e^{-2\pi i x \xi} 
			\rd{x} \\
			= & \int_{\mathbb{R}}	\big(\frac{\rd}{\rd x}\big)^k\big((|x|^a-|x|^4)\Phi(x)\big)  e^{-2\pi i x \xi} 
			\rd{x}\,. \\
	\end{split}\end{equation}
	Then the conclusion can obtained by a decomposition similar to \eqref{Deltaa4}. In fact, 	if $a$ is not an integer, then
	\begin{equation}\label{Deltaa4d1_2}
		 \int_{|x|\le 1}	\big(\frac{\rd}{\rd x}\big)^k\big((|x|^a-|x|^4)\Phi(x)\big)\Phi(2x)  e^{-2\pi i x \xi} 
		\rd{x} = \prod_{j=0}^{k-1}(a-j) \cdot  \cF[|x|^{a-k}(\sgn x)^k\Phi(2x)](\xi)\,,
	\end{equation}
	where $-1<a-k<0$. If $k$ is even, then the last expression is bounded by $\eta_{a,1}(1+|\xi|)^{-1+k-a}$ similar as its $d\ge 2$ analog. If $k$ is odd (which includes the case when $a$ is close 4, $k=5$), we first use the formula
	\begin{equation}
		\cF[|x|^{a-k}(\sgn x)] = -2i \pi^{k-a-\frac{1}{2}}\frac{\Gamma(\frac{2-k+a}{2})}{(1+a-k)\Gamma(\frac{k-a-1}{2})}|\xi|^{-1+k-a}\sgn\xi\,,
	\end{equation}
	which can be obtained by \eqref{thm_powerF_1} via differentiating in $x$. Then
	\begin{equation}
		\cF[|x|^{a-k}(\sgn x)^k\Phi(2x)](\xi) = -i \pi^{k-a-\frac{1}{2}}\frac{\Gamma(\frac{2-k+a}{2})}{(1+a-k)\Gamma(\frac{k-a-1}{2})}\big(|\xi|^{-1+k-a}\sgn\xi\big)* \hat{\Phi}(\frac{\xi}{2})\,,
	\end{equation}
	is bounded by $C_1|\xi|^{-1+k-a}$ uniformly in $a$ for $-1<a-k<0$. In fact, one technical issue here is that $|\xi|^{-1+k-a}\sgn\xi$ has a sharp singularity near 0 when $a-k$ is close to 0. It does not affect the uniform estimate for the convolution because $|\xi|^{-1+k-a}\sgn\xi$ is odd and $\hat{\Phi}(\frac{\xi}{2})$ is smooth.
	
	Then we see that \eqref{Deltaa4d1_2} is again bounded by $\eta_{a,1}(1+|\xi|)^{-1+k-a}$, with $\eta_{a,1}=O(a-4)$ when $a$ is close to 4 since the product in \eqref{Deltaa4d1_2} contains a factor $(a-4)$.
	
	If $a$ is an integer, then $k=1+a$. If $a$ is even, then the LHS of \eqref{Deltaa4d1_2} is zero. If $a$ is odd, then $(\frac{\rd}{\rd x})^k\big((|x|^a-|x|^4)\Phi(x)\big)\Phi(2x) = 2\cdot a! \cdot \delta(x)$ and thus the LHS of \eqref{Deltaa4d1_2} is $2\cdot a!$, which is again bounded by $\eta_{a,1}(1+|\xi|)^{-1+k-a}$ with $\eta_{a,1}=2\cdot a!$.
	
	The term $ \int_{1/2 < |x|\le 2}	(\frac{\rd}{\rd x})^k\big((|x|^a-|x|^4)\Phi(x)\big)\cdot(1-\Phi(2x))  e^{-2\pi i x \xi} 
	\rd{x}$ can be handle in the same way as the $d\ge 2$ case, and the conclusion for $d= 1$ follows.
		
\end{proof}

\begin{proof}[Proof of Theorem \ref{thm_main2}]
	
	Theorem \ref{thm_existab} shows that $R_*(b,4)$ is an upper bound of minimizer size for any $W_{a,b}$ with $a\ge 4$. 
	
	We claim that for $a_+(b)>4$ sufficiently close to 4, we have 
	\begin{equation}\label{claim_RLIC}
		R_{\textnormal{LIC}}[W_{a,b}] \ge R_*(b,4) +1 =: R,\quad \forall 4<a \le a_+(b)\,,
	\end{equation}
	which would imply the conclusion due to Lemma \ref{lem_LICR}.
	
	By Theorem \ref{thm_RLICab}, it suffices to prove that for any nontrivial signed measure $\mu$ supported on $\overline{B(0;R)}$ with $E_{a,b}[|\mu|]<\infty$ and $\int_{\mathbb{R}^d}\rd{\mu(\bx)}=\int_{\mathbb{R}^d}\bx\rd{\mu(\bx)}=0$, there holds
	\begin{equation}
		E_{a,b}[\mu] > 0\,.
	\end{equation}
	Denote
	\begin{equation}
		\tilde{W}_{a,b} = (2R)^{a-b}\frac{|\bx|^a}{a} - \frac{|\bx|^b}{b}\,,
	\end{equation}
	and its associated energy $\tilde{E}_{a,b}$. By rescaling, it suffices to prove that for any $\mu$ supported on $\overline{B(0;\frac{1}{2})}$ with $\tilde{E}_{a,b}[|\mu|]<\infty$ and $\int_{\mathbb{R}^d}\rd{\mu(\bx)}=\int_{\mathbb{R}^d}\bx\rd{\mu(\bx)}=0$, there holds
	\begin{equation}
		\tilde{E}_{a,b}[\mu] > 0\,.
	\end{equation}
	For such $\mu$, we may rewrite $\tilde{E}_{a,b}[\mu]$ as
	\begin{equation}\begin{split}
		\tilde{E}_{a,b}[\mu] = & E_b[\mu] + \frac{(2R)^{a-b}}{2a}\int_{|\bx|\le 1/2}\int_{|\by|\le 1/2}|\bx-\by|^a\rd{\mu(\by)}\rd{\mu(\bx)} \\
		= & E_b[\mu] +  \frac{(2R)^{a-b}}{2a}\int_{|\bx|\le 1/2}\int_{|\by|\le 1/2}(|\bx-\by|^a-|\bx-\by|^4)\Phi(\bx-\by)\rd{\mu(\by)}\rd{\mu(\bx)} \\ & + \frac{(2R)^{a-b}}{2a}\int_{|\bx|\le 1/2}\int_{|\by|\le 1/2}|\bx-\by|^4\rd{\mu(\by)}\rd{\mu(\bx)} \,,\\
	\end{split}\end{equation}
	where $\Phi$ is given by the statement of Lemma \ref{lem_a4}, {and $E_b[\mu]$ denotes the energy \eqref{E} with interaction potential $W(\bx)=-\frac{|\bx|^b}{b}$}. The last integral is nonnegative by the (nonstrict) LIC property of the interaction potential $|\bx|^4$ (c.f. \cite[Theorem 2.4]{Lop19}). Then notice that the potential $(|\bx|^a-|\bx|^4)\Phi(\bx)$ is Fourier representable at level 0 due to Theorem \ref{thm_convFou1}, whose Fourier transform is controlled by Lemma \ref{lem_a4}. Therefore we get
	\begin{equation}\begin{split}
		\tilde{E}_{a,b}[\mu] \ge \frac{1}{2}\int_{\mathbb{R}^d\backslash\{0\}} \Big(\cpf(-b)|\xi|^{-d-b} -  \frac{(2R)^{a-b}}{a}\eta_a (1+|\xi|)^{-d-a}\Big)|\hat{\mu}(\xi)|^2\rd{\xi}\,,
	\end{split}\end{equation}
	where $\eta_a = O(a-4)$ when $a$ is close to 4. Since $a>4$ and $-d<b<2$, we see that $|\xi|^{-d-b} \ge (1+|\xi|)^{-d-b} \ge (1+|\xi|)^{-d-a}$ for any $\xi\ne 0$. Also, $\cpf(-b)>0$, and $\frac{(2R)^{a-b}}{a}$ is uniformly bounded when $a>4$ is close to 4. Therefore, using $\eta_a = O(a-4)$, one can take $a_+(b)$ sufficiently close to 4 to guarantee that $\frac{(2R)^{a-b}}{a}\eta_a \le \frac{1}{2}\cpf(-b)$ for any $4<a\le a_+(b)$, and thus $\tilde{E}_{a,b}[\mu]>0$.

\end{proof}

\begin{remark}
	We believe that the approach in this section can also be used to prove Theorem \ref{thm_main1}. However, this approach does not yield a clean description for the precise value of the LIC radius, as discussed in Remark \ref{rmk_sharp}.
\end{remark}

\section*{Acknowledgment}

The author would like to thank Rupert L. Frank and Ryan W. Matzke for helpful discussions.

\bibliographystyle{alpha}
\bibliography{minimizer_book_bib.bib}

\end{document}